\documentclass[usenames,dvipsnames]{amsart}
\usepackage{amssymb,amsmath,latexsym,times,tikz,mathtools,hyperref,multirow,float}
\usetikzlibrary{patterns}

\hypersetup{colorlinks=true, citecolor=blue, linkcolor=blue}

\newtheorem{thm}{Theorem}[section]

\newtheorem{cor}[thm]{Corollary}
\newtheorem{prop}[thm]{Proposition}
\newtheorem{lemma}[thm]{Lemma}

\newtheorem{conj}[thm]{Conjecture}
\newtheorem{open}[thm]{Open Question}
\newtheorem{sublemma}{}[thm]

\DeclareMathOperator{\Ex}{Ex}

\newcommand{\del}{\backslash}
\newcommand{\ba}{\backslash}

\newcommand{\cn}{contradiction}
\newcommand{\btu}{\bigtriangleup}

\title[Delta-matroids as subsystems of sequences of Higgs
lifts]{Delta-matroids as subsystems of sequences of Higgs lifts}
\author[J.~Bonin]{Joseph E.~Bonin} \address[J.~Bonin]
{Department of Mathematics\\ The George Washington University\\
  Washington, D.C. 20052, USA} \email[J.~Bonin] {jbonin@gwu.edu}
\author[C.~Chun] {Carolyn Chun} \address[C.~Chun]
{Department of Mathematics\\
  United States Naval Academy\\
  Annapolis, MD, 21402, USA} \email[C.~Chun] {chun@usna.edu}
\author[S.~Noble] {Steven D.\ Noble} \address[S.~Noble] {Department of
  Economics, Mathematics and Statistics \\ Birkbeck, University of
  London \\ London WC1E 7HX, United Kingdom} \email[S.~Noble]
{s.noble@bbk.ac.uk} \date{\today}

\begin{document}

\begin{abstract}
  In~\cite{MR843383}, Tardos studied special delta-matroids obtained
  from sequences of Higgs lifts; these are the full Higgs lift
  delta-matroids that we treat and around which all of our results
  revolve. We give an excluded-minor characterization of the class of
  full Higgs lift delta-matroids within the class of all
  delta-matroids, and we give similar characterizations of two other
  minor-closed classes of delta-matroids that we define using Higgs
  lifts.  We introduce a minor-closed, dual-closed class of Higgs lift
  delta-matroids that arise from lattice paths. It follows from
  results of Bouchet that all delta-matroids can be obtained from full
  Higgs lift delta-matroids by removing certain feasible sets; to
  address which feasible sets can be removed, we give an
  excluded-minor characterization of delta-matroids within the more
  general structure of set systems.  Many of these excluded minors
  occur again when we characterize the delta-matroids in which the
  collection of feasible sets is the union of the collections of bases
  of matroids of different ranks, and yet again when we require those
  matroids to have special properties, such as being paving.
\end{abstract}

\maketitle

\section{Introduction}

A \emph{set system} is a pair $S=(E,\mathcal{F})$, where $E$, or
$E(S)$, is a set, called the \emph{ground set}, and $\mathcal{F}$, or
$\mathcal{F}(S)$, is a collection of subsets of $E$.  (All set systems
in this paper have finite ground sets.)  The members of $\mathcal{F}$
are the \emph{feasible sets}.  We say that $S$ is \emph{proper} if
$\mathcal{F}\ne \emptyset$, and that $S$ is \emph{even} if $|X|-|Y|$
is even for all $X,Y\in\mathcal{F}$.  A matroid $M$ has many
associated set systems with $E=E(M)$ since we can take $\mathcal{F}$
to be, for example, the set $\mathcal{B}(M)$ of its bases, or the set
of its independent sets, or the set of its circuits; the first two are
always proper.  The first is of most interest here since the
definition of a delta-matroid can be motivated by an exchange property
that the bases of any matroid $M$ satisfy, namely, for any
$B_1,B_2\in\mathcal{B}(M)$ and for each element $x\in B_1-B_2$, there
is a $y\in B_2-B_1$ for which $B_1\triangle\{x,y\}\in\mathcal{B}(M)$.
To get the definition of a delta-matroid, replace the set differences
by symmetric differences.  Thus, as introduced by Bouchet
in~\cite{ab1}, a \emph{delta-matroid} is a proper set system
$D=(E,\mathcal{F})$ for which $\mathcal{F}$ satisfies the
\emph{delta-matroid symmetric exchange axiom}:
\begin{quote}
  (SE) \
  for all triples $(X,Y,u)$ with $X$ and $Y$ in $\mathcal{F}$
  and $u\in X\triangle Y $, there is a $v \in X\triangle Y$ (perhaps
  $u$ itself) such that $X\triangle \{u,v\}$ is in $\mathcal{F}$.
\end{quote}
Just as there is a mutually-enriching interplay between matroid theory
and graph theory, the theory of delta-matroids has substantial
connections with the theory of embedded graphs; see~\cite{CMNR,CMNR2}.

Naturally, there are strong links between matroids and delta-matroids;
below we cite several that are relevant in this paper.  First, for a
delta-matroid $D$, let $\max(\mathcal{F}(D))$ be the collection of
sets in $\mathcal{F}(D)$ that have the largest cardinality among sets
in $\mathcal{F}(D)$, and define $\min(\mathcal{F}(D))$ similarly.  An
easy application of property (SE) shows that each of
$\max(\mathcal{F}(D))$ and $\min(\mathcal{F}(D))$ is the collection of
bases of a matroid on $E$; we denote these matroids by $D_{\max}$ and
$D_{\min}$, respectively, and call them the \emph{maximal} and
\emph{minimal matroids} of $D$.

A matroid $Q$ on $E$ is a \emph{quotient} of a matroid $L$ on $E$, or
$L$ is a \emph{lift} of $Q$, if there is a matroid $M$ and a subset
$X$ of $E(M)$ for which $M\del X=L$ and $M/X=Q$.  The following
connection between $D_{\min}$ and $D_{\max}$ was proven by
Bouchet~\cite[Theorem 3.3]{multi2}.

\begin{prop}\label{prop:bouchetquotient}
  For any delta-matroid $D$, the matroid $D_{\min}$ is a quotient of
  $D_{\max}$.
\end{prop}

This result and the following property of $D_{\min}$ and $D_{\max}$
are important in our work.

\begin{lemma}\label{lem:containment}
  If $X$ is any feasible set in a delta matroid $D$, then there are
  bases $B'$ of $D_{\min}$ and $B$ of $D_{\max}$ with
  $B'\subseteq X\subseteq B$.
\end{lemma}

In other words, in a delta-matroid, the minimal feasible sets with
respect to inclusion are exactly the minimal feasible sets with
respect to cardinality, and likewise for maximal.  Lemma
\ref{lem:containment} follows from \cite[Property 4.1]{rep} and is
easy to prove directly.

The converse of Proposition~\ref{prop:bouchetquotient} is true.  One
way to show it is to show that if $Q$ is a quotient of $L$, with both
matroids on the set $E$, and if we let $\mathcal{F}$ be the set of all
subsets $X$ of $E$ for which there are bases $B'\in\mathcal{B}(Q)$ and
$B\in\mathcal{B}(L)$ with $B'\subseteq X\subseteq B$, then
$(E,\mathcal{F})$ is a delta-matroid.  Such delta-matroids were
studied by Tardos in~\cite{MR843383}; she called them generalized
matroids.  In Section~\ref{sec:hl}, we interpret the construction of
these special delta-matroids using the Higgs lifts of $Q$ toward $L$;
thus, we call such delta-matroids full Higgs lift delta-matroids.  We
consider beginning with a full Higgs lift delta-matroid and removing
all of the feasible sets of certain cardinalities.  We call this a
Higgs lift delta-matroid, or an even Higgs lift delta-matroid when all
of the feasible sets of one parity are removed.  (See
Proposition~\ref{new31}.)  We give an excluded-minor characterization
of Higgs lift delta-matroids (Theorem~\ref{extraexhiggs}), as well as
counterparts in the full case and in the even case.  In
Section~\ref{sec:lp}, we introduce Higgs lift delta-matroids that
arise from lattice paths.

Lemma~\ref{lem:containment} says that any delta-matroid can be
obtained from a full Higgs lift delta-matroid by discarding some of
the feasible sets.  It is natural to ask what restrictions there are
on the sets that we remove.  This issue is addressed in
Section~\ref{sec:exmin}, where we give an excluded-minor
characterization of delta-matroids within the broader structure of set
systems.  We address the corresponding issues for even delta-matroids,
for matroids, and for binary delta-matroids.

  For a delta-matroid $D$ and any integer $i$ with
  $r(D_{\min})\leq i\leq r(D_{\max})$, let $N_i$ be the set system
  $(E,\{F\in\mathcal{F} \,:\, |F|=i\})$.  If $D$ is a Higgs lift
  delta-matroid, then each proper set system
  $N_i$ is a matroid, but this need not be
  true for other delta-matroids.  In Section~\ref{sec:msdm}, we
  characterize the delta-matroids $D$ for which each $N_i$ is a
  matroid, as well as, for instance, when each $N_i$ is a paving
  matroid or a sparse paving matroid.

We follow the notation and terminology for matroids that is used in~\cite{oxley}.  In the next section, we review some key points about
delta-matroids, as well as some of the more specialized matroid topics
that play roles throughout this paper.

\section{Background}

Two set systems $S=(E,\mathcal{F})$ and $S'=(E',\mathcal{F}')$ are
\emph{isomorphic} if there is a bijection $\phi:E\to E'$ so that, for
all $A\subseteq E$, we have $A\in \mathcal{F}$ if and only if
$\phi(A)\in\mathcal{F}'$.

\subsection{Minors and twists of set systems}
Let $S=(E,\mathcal{F})$ be a proper set system.  An element $e\in E$ is a
\emph{loop} of $S$ if no set in $\mathcal{F}$ contains $e$.  If $e$ is
in every set in $\mathcal{F}$, then $e$ is a \emph{coloop}.  If $e$ is
not a loop, then the \emph{contraction of $e$ from $S$}, written
$S/e$, is given by
\[S/e = (E-e, \{F-e:e\in F\in\mathcal{F}\}).\]  (As in matroid theory,
we usually omit set brackets from singleton sets.)  If $e$ is not a
coloop, then the \emph{deletion of $e$ from $S$}, written $S\ba e$, is
given by
\[S\ba e = (E-e,\{F\subseteq E-e:F\in\mathcal{F}\}).\] If $e$ is a
loop or a coloop, then one of $S/e$ and $S\ba e$ has already been
defined, so we can set $S/e=S\ba e$.  Any sequence of deletions and
contractions, starting from $S$, gives a set system $S'$, called a
\emph{minor} of $S$.  Each minor of $S$ is a proper set system.  Note
that if $S$ is even, then so are its minors.

A collection $\mathcal C$ of proper set systems is \emph{minor closed}
if every minor of every member of $\mathcal C$ is in $\mathcal C$.
Given such a collection $\mathcal C$, a proper set system $S$ is an
\emph{excluded minor} for $\mathcal{C}$ if $S\notin\mathcal{C}$ and
all other minors of $S$ are in $\mathcal{C}$.  A proper set system
belongs to $\mathcal C$ if and only if none of its minors is
isomorphic to an excluded minor for $\mathcal C$.  Thus, the excluded
minors determine $\mathcal C$; they are the minor-minimal obstructions
to membership in $\mathcal{C}$.

The order in which elements are deleted or contracted can matter
since, for instance, contracting an element $e$ can turn a non-loop of
$S$ into a loop of $S/e$.  For example, if
$S=(\{a,b,c,d\},\{\{a,b\},\{c,d\}\})$, then $c$ is a loop of $S/a$ and
$S/a/c=(\{b,d\},\{b\})$, whereas $a$ is a loop of $S/c$ and
$S/c/a=(\{b,d\},\{d\})$.  However, for disjoint subsets $X$ and $Y$ of
$E$, if some set in $\mathcal{F}$ is disjoint from $X$ and contains
$Y$, then the deletions and contractions in $S\ba X/Y$ can be done in
any order, and
\begin{equation}\label{eq:minor}
  S\ba X/Y=(E-(X\cup Y),\{F-Y\,:\,F\in\mathcal{F} \text{ and
  }Y\subseteq F\subseteq E-X\}).
\end{equation}
We next show that all minors of a proper set system are of this type.

\begin{lemma}\label{minorisminor}
  For any minor $S'$ of a proper set system $S=(E,\mathcal{F})$, there
  are disjoint subsets $X$ and $Y$ of $E$ for which \emph{(i)} some
  set in $\mathcal{F}$ is disjoint from $X$ and contains $Y$, and
  \emph{(ii)} $S'=S\del X/Y$.  Thus, $S'$ is given by equation
  \emph{(\ref{eq:minor})}.
\end{lemma}

\begin{proof}
  Suppose we get $S'$ from $S$ by, for each of $e_1,e_2,\ldots,e_k$ in
  turn, either deleting or contracting $e_i$, giving the sequence of
  minors $S_0=S,S_1,S_2,\ldots ,S_k=S'$.  Let $X$ be the set of
  elements $e_i$ in $\{e_1,e_2,\ldots,e_k\}$ that satisfy at least one
  of the following conditions:
  \begin{enumerate}
  \item $e_i$ is a loop of $S_{i-1}$ (so $S_i = S_{i-1}\del e_i$), or
  \item $e_i$ is not a coloop of $S_{i-1}$ and $S_i=S_{i-1}\del e_i$.
  \end{enumerate}
  Let $Y= \{e_1,e_2,\ldots,e_k\}-X$, so for each $e_j\in Y$, either
  $e_j$ is a coloop of $S_{j-1}$ (so $S_j = S_{j-1}/e_j$), or $e_j$ is
  not a loop of $S_{j-1}$ and $S_j=S_{j-1}/e_j$.  Thus, to get $S'$
  from $S$, for $e_1,e_2,\ldots,e_k$ in turn, delete each $e_i\in X$
  and contract each $e_j\in Y$.  By the definitions of $X$, $Y$, and
  these operations, if $F$ is a feasible set of $S'$, then
  $F\cup Y\in\mathcal{F}$, so assertion (i) holds.  Given the remarks
  above, the remaining assertions now follow.
\end{proof}

Bouchet and Duchamp~\cite{rep} showed that if $S$ is a delta-matroid
and $S'=S\ba X/Y$, then $S'$ is a delta-matroid and $S'$ is
independent of the order of the deletions and contractions.

For $A\subseteq E$, the \emph{twist of $S$ on $A$}, which is also
called the \emph{partial dual of $S$ with respect to $A$}, denoted
$S*A$, is given by
\[S*A=(E,\{F\btu A\,:\,F\in\mathcal{F}\}).\] Note that
$S/e=(S * e)\ba e$ and $(S*A)*A=S$.  The \emph{dual} $S^*$ of $S$ is
$S*E$.  Note that twists of even set systems are even.  However, apart
from the dual, the twists of a matroid
$\bigl(E(M),\mathcal{B}(M)\bigr)$ are generally not matroids, as
discussed in~\cite[Theorem~3.4]{CMNR2}.

\subsection{Quotients, lifts, and Higgs lifts}
We will use the following result about quotients, which is well known
(see, e.g.,~\cite{Brylawski,oxley}).

\begin{lemma}\label{lem:quotviabases}
  For matroids $Q$ and $L$ on $E$, the statements below are
  equivalent.
  \begin{enumerate}
  \item The matroid $Q$ is a quotient of $L$.
  \item The matroid $L^*$ is a quotient of $Q^*$.
  \item\label{lemlabel:quocir} Each circuit of $L$ is a union of
    circuits of $Q$.
  \item\label{lemlabel:quobasis} For each basis $B$ of $L$ and element
    $e\in E-B$, there is a basis $B'$ of $Q$ with $B'\subseteq B$ and
    \[\{f\,:\, (B'\cup e)-f \text{ is a basis of } Q\}\subseteq
    \{f\,:\, (B\cup e)-f \text{ is a basis of } L\}.\]
  \end{enumerate}
\end{lemma}

We will use Higgs lifts, for which we recall only the background we
need.  (See~\cite{splice, Brylawski, Strong} for more about this
construction.)  Let $Q$ be a quotient of $L$ on $E$ and set
$k=r(L)-r(Q)$.  For each integer $i$ with $0\leq i\leq k$, the
function $r_i$ that is defined by
\begin{equation}\label{Higgsrank}
  r_i(X) = \min\{r_{Q}(X)+i,\,r_{L}(X)\},
\end{equation}
for $X\subseteq E$, is the rank function of a matroid on $E$; this
matroid is the \emph{$i$-th Higgs lift of $Q$ toward $L$} and is
denoted $H^i_{Q,L}$.  Its bases are the sets of size $r(Q)+i$ that
span $Q$ and are independent in $L$, or, equivalently, contain a basis
of $Q$ and are themselves contained in a basis of $L$.  It follows that
if $0\leq i\leq j\leq k$, then $H^j_{Q,L}$ is the $(j-i)$-th Higgs
lift of $H^i_{Q,L}$ toward $L$.  The matroid $H^i_{Q,L}$ is the freest
(i.e., greatest in the weak order) quotient of $L$ that has $Q$ as a
quotient and has rank $r(Q)+i$.  Higgs lifts commute with minors and
duals, as we state next.  (See~\cite[Propositions 2.2 and 2.6]{splice}
for proofs.)  So that we do not need to restrict $i$ and $j$ below, as
is common we set $H^i_{Q,L}$ to $L$ when $i>k$, and to $Q$ when $i<0$.

\begin{lemma}\label{lem:higgsdual}
  If $Q$ is a quotient of $L$ and $i+j=r(L)-r(Q)$, then $(H^i_{Q,L})^*
  =H^j_{L^*,Q^*}$.  Also, if $X\subseteq E$, then $(H^i_{Q,L})|X
  =H^i_{Q|X,L|X}$ and $(H^i_{Q,L})/X =H^{i-t}_{Q/X,L/X}$ where
  $t=r_L(X)-r_Q(X)$.
\end{lemma}

\section{Higgs lift delta-matroids}\label{sec:hl}

It is often useful to view a simple graph on $n$ vertices as a
subgraph of the maximal such graph, $K_n$.  Similarly, a rank-$r$
simple matroid that is representable over $GF(q)$ can be seen as a
restriction of the maximal such matroid, $PG(r-1,q)$.  In that spirit,
by the next two results we can view each delta-matroid $D$ as coming
from the maximal delta-matroid that has the same minimal and maximal
matroids as $D$.  These maximal delta-matroids correspond to the case
$K=\{0,1,\ldots,k\}$ in the next result.  This result shows that the
converse of Proposition~\ref{prop:bouchetquotient} holds.

\begin{prop}
\label{new31}
Fix a matroid $L$ on $E$ and a quotient $Q$ of $L$. Set $k=r(L)-r(Q)$
and let $K$ be a subset of $\{0,1,2,\ldots,k\}$ for which
$\{0,1,2,\ldots,k\}-K$ contains no pair of consecutive integers.  Then
the union
\[\mathcal{F}=\bigcup_{i\in K}\mathcal{B}(H^i_{Q,L})\] of
the sets of bases of the Higgs lifts $H^i_{Q,L}$ of $Q$ towards $L$,
indexed by the elements of $K$, is the set of feasible sets of a
delta-matroid on $E$.
\end{prop}

\begin{proof}
  With the first part of Lemma~\ref{lem:higgsdual} and the observation
  that $H^i_{Q,L},H^{i+1}_{Q,L},\ldots,H^j_{Q,L}$ are the Higgs lifts
  of $H^i_{Q,L}$ toward $H^j_{Q,L}$, we may assume that
  $\{0,k\}\subseteq K$, and it suffices to check property (SE) for all
  triples $(X,Y,u)$, where $X\in\mathcal{B}(Q)$ and
  $Y\in\mathcal{B}(L)$ and $u\in X\btu Y$.  Bases of $L$ span $Q$, so
  $Y$ spans $Q$.  If $u\in X-Y$, then, since $Y$ spans $Q$, there is a
  $v\in Y-X$ for which $(X-u)\cup v$ is a basis of $Q$, so property
  (SE) holds.  Now assume that $u\in Y-X$.  Note that by the
  hypothesis, $K$ contains either $1$ or $2$.  First assume that
  $X\cup u$ is independent in $L$.  Thus, $X\cup u$ is a basis of
  $H^1_{Q,L}$, so taking $v=u$ verifies property (SE) if $1$ is in
  $K$.  Note that $X\cup u$ is independent in $H^{2}_{Q,L}$ and $Y$
  spans $H^{2}_{Q,L}$, so there is a $v\in Y-(X\cup u)$ with
  $X\cup \{u,v\}\in \mathcal{B}(H^{2}_{Q,L})$, so property (SE) holds
  if $2$ is in $K$.  Now assume that $X\cup u$ is dependent in $L$, so
  it contains a unique circuit, say $C$, of $L$.  Since $Y$ is a basis
  of $L$, we have $C\not\subseteq Y$, so fix a $v\in C-Y$.  By
  part~(\ref{lemlabel:quocir}) of Lemma~\ref{lem:quotviabases}, $C$ is
  a union of circuits of $Q$, and since $X$ is a basis of $Q$, the set
  $X\cup u$ contains a unique circuit of $Q$, so $C$ is a circuit of
  $Q$.  Now $v\in X-Y$ and $(X\cup u)-v$ is a basis of $Q$, as needed.
\end{proof}

We call the delta-matroids identified in Proposition~\ref{new31}
\emph{Higgs lift delta-matroids}.  If $K=\{0,1,2,\ldots ,k\}$, we have
the \emph{full Higgs lift delta-matroid of the pair $(Q,L)$}; they
were studied by Tardos~\cite{MR843383}, who called them generalized
matroids, and more recently in \cite{saturateddelta}, where they are
called saturated delta-matroids. If $k$ and all elements of $K$ are
even, we have the \emph{even Higgs lift delta-matroid of the pair
  $(Q,L)$}.

It is straightforward to obtain the following characterization of
the feasible sets in a Higgs lift delta-matroid.

\begin{lemma}\label{higgsfeas}
  A delta-matroid $D=(E,\mathcal{F})$ is a Higgs lift delta-matroid if
  and only if, for every set $F\subseteq E$, one of the following
  holds:
  \begin{enumerate}
  \item no set in $\mathcal{F}$ has cardinality $|F|$ or
  \item $F\in\mathcal{F}$ exactly when there exist sets
    $A\in\mathcal{B}(D_{\min})$ and $B\in\mathcal{B}(D_{\max})$ such
    that $A\subseteq F\subseteq B$.
  \end{enumerate}
\end{lemma}

The next result follows from Lemma~\ref{lem:containment} and the
description of the bases of Higgs lifts.

\begin{cor}\label{cor32}
  If $X$ is a feasible set in a delta-matroid $D$ and
  $i=|X|-r(D_{\min})$, then $X$ is a basis of the $i$-th Higgs lift of
  $D_{\min}$ toward $D_{\max}$.  Thus, $D$ is obtained from the full
  Higgs lift delta-matroid of the pair $(D_{\min},D_{\max})$ by
  removing some feasible sets that are not in
  $\mathcal{B}(D_{\min})\cup \mathcal{B}(D_{\max})$.
\end{cor}

Theorem~\ref{exdelta} addresses the question of which feasible
sets of the Higgs lift delta-matroid of a pair $(Q,L)$ can be removed
to yield delta-matroids.

Now we give an excluded-minor characterization of Higgs lift
delta-matroids.  We will use the following seven delta-matroids:
\begin{itemize}
\item $U_1 = (\{a,b\},\bigl\{\emptyset, \{a\},\{a,b\} \bigr\})$,
\item
  $U_2 = (\{a,b,c\},\bigl\{\emptyset, \{c\},\{a,b\},\{a,b,c\}
  \bigr\})$,
\end{itemize}
and, for $3\leq i\leq 7$, the even delta-matroid $U_i$ has ground set
$E=\{a,b,c,d\}$ and its feasible sets are $\emptyset$, $E$, and the
$2$-element sets given by the edges of the graph $G_i$ in
Figure~\ref{fig:Ui}.

\begin{figure}
  \centering
\begin{tikzpicture}
\tikzstyle{vertex}=[circle, draw, inner sep=0pt, minimum size=6pt]
\newcommand{\vertex}{\node[vertex]}

\vertex[fill] (a3) at (0,1) [label=above:{\small $a$}] {};
\vertex[fill] (b3) at (1,1) [label=above:{\small $b$}] {};
\vertex[fill] (c3) at (1,0) [label=below:{\small \rule{0pt}{7pt}$c$}] {};
\vertex[fill] (d3) at (0,0) [label=below:{\small $d$}] {};
\path[thick] (a3) edge (b3)
(c3) edge (d3)	;
\node at (0.5,-0.95) {\small $G_3$};
\vertex[fill] (a4) at (2.3,1) [label=above:{\small $a$}] {};
\vertex[fill] (b4) at (3.3,1) [label=above:{\small $b$}] {};
\vertex[fill] (c4) at (3.3,0) [label=below:{\small \rule{0pt}{7pt}$c$}] {};
\vertex[fill] (d4) at (2.3,0) [label=below:{\small $d$}] {};
\path[thick] (a4) edge (b4)  (b4) edge (c4)
(c4) edge (d4)	;
\node at (2.8,-0.95) {\small $G_4$};
\vertex[fill] (a5) at (4.6,1) [label=above:{\small $a$}] {};
\vertex[fill] (b5) at (5.6,1) [label=above:{\small $b$}] {};
\vertex[fill] (c5) at (5.6,0) [label=below:{\small \rule{0pt}{7pt}$c$}] {};
\vertex[fill] (d5) at (4.6,0) [label=below:{\small $d$}] {};
\path[thick] (a5) edge (b5)  (b5) edge (c5)  (a5) edge (d5)
(c5) edge (d5)	;
\node at (5.1,-0.95) {\small $G_5$};
\vertex[fill] (a6) at (6.9,1) [label=above:{\small $a$}] {};
\vertex[fill] (b6) at (7.9,1) [label=above:{\small $b$}] {};
\vertex[fill] (c6) at (7.9,0) [label=below:{\small \rule{0pt}{7pt}$c$}] {};
\vertex[fill] (d6) at (6.9,0) [label=below:{\small $d$}] {};
\path[thick] (a6) edge (b6)  (b6) edge (c6)  (a6) edge (c6)
(a6) edge (d6)	;
  \node at (7.4,-0.95) {\small $G_6$};
\vertex[fill] (a7) at (9.2,1) [label=above:{\small $a$}] {};
\vertex[fill] (b7) at (10.2,1) [label=above:{\small $b$}] {};
\vertex[fill] (c7) at (10.2,0) [label=below:{\small \rule{0pt}{7pt}$c$}] {};
\vertex[fill] (d7) at (9.2,0) [label=below:{\small $d$}] {};
\path[thick] (a7) edge (b7)  (b7) edge (c7)  (a7) edge (c7) (c7) edge (d7)
(a7) edge (d7)	;
  \node at (9.7,-0.95) {\small $G_7$};
\end{tikzpicture}
\caption{The graphs whose edges give the proper, nonempty feasible
  sets of $U_3$, $U_4$, $U_5$, $U_6$, and $U_7$, respectively.}
  \label{fig:Ui}
\end{figure}
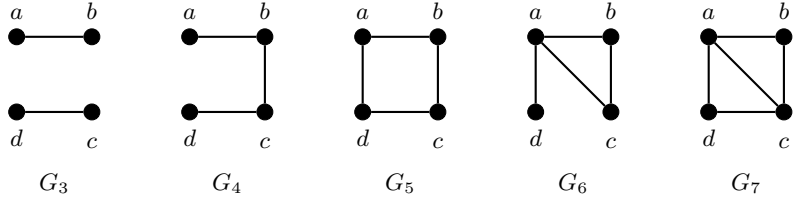

\begin{thm}\label{extraexhiggs}
  A delta-matroid is a Higgs lift delta-matroid if and only if it has
  no minor isomorphic to any of $U_1,U_2,\dots ,U_7$.
\end{thm}

The proof of the theorem is postponed until Section~\ref{sec:exmin}.
This result gives part of the next corollary; the rest is easy to
check.  The duality assertion uses the first part of
Lemma~\ref{lem:higgsdual}.

\begin{cor}
  The classes of Higgs lift delta-matroids,
  full Higgs lift delta-matroids, and even Higgs lift
  delta-matroids are closed under minors and duals.
\end{cor}

Let $S_2$ be the delta-matroid
  $(\{a,b\},\{\emptyset ,\{a,b\}\})$.  We now characterize full Higgs
  lift delta-matroids and even Higgs lift delta-matroids by their
  excluded minors.

\begin{cor}
\label{exfull}
A delta-matroid is a full Higgs lift delta-matroid if and only if
it contains no minor isomorphic to $U_1$ or $S_2$.
\end{cor}
\begin{proof}
  It is straightforward to check that $U_1$ and $S_2$ are excluded
  minors for the class of full Higgs lift delta-matroids.

  Suppose that the delta-matroid $D=(E,\mathcal F)$ is not a full
  Higgs lift delta-matroid. If $D$ is not a Higgs lift delta-matroid,
  then it has a minor in $\{U_1,U_2,\ldots,U_7\}$ and each of
  $U_2,U_3,\ldots,U_7$ has a minor isomorphic to $S_2$. Suppose that
  $D$ is a Higgs lift delta-matroid but not a full Higgs lift
  delta-matroid.  For $i$ with $0\leq i\leq r(D_{\max})-r(D_{\min})$,
  let $N_i$ be the set system
  $(E,\{F\in\mathcal F: |F|=i+r(D_{\min})\})$. Then for some $i$ with
  $0<i<r(D_{\max})-r(D_{\min})$, the set system $N_i$ is
  improper. Both $N_{i-1}$ and $N_{i+1}$ must be proper in order for
  $D$ to be a delta-matroid. Choose bases $B_Q$ and $B_L$ of
  $D_{\min}$ and $D_{\max}$ respectively with $B_Q\subseteq B_L$. Then
  there are sets $X$ and $Y$ belonging to $N_{i-1}$ and $N_{i+1}$
  respectively, satisfying
  $B_Q \subseteq X \subseteq Y \subseteq B_L$. So $D/X\setminus (E-Y)$
  is isomorphic to $S_2$.
 \end{proof}

 The next corollary follows because a delta-matroid is both even and a
 Higgs lift delta-matroid if and only if it is an even Higgs lift
 delta-matroid.

\begin{cor}
\label{exeven}
An even delta-matroid is an even Higgs lift delta-matroid if and only
if it contains no minor isomorphic to $U_3,U_4,U_5,U_6$, or $U_7$.
\end{cor}

\section{Lattice path delta-matroids}\label{sec:lp}

In this section we define a class of full Higgs lift delta-matroids
using lattice paths.  This is a natural direction in which to extend
the theory of lattice path matroids, which has proven to be a rich
vein; for instance, see~\cite{lpmfacial, MR2718679, lpm2, lpmfm,
  MR3413585, lpmtoric, lpmpolytope, lpmTPineq, tennis, lpmcomput,
  MR2802156, MR2578897}.  The concrete nature of the delta-matroids
defined below may help readers get a better handle on delta-matroids,
and it may suggest new avenues of investigation.

We first recall lattice path matroids from~\cite{lpm1}.  (See
Figure~\ref{fig:lpmexamples} for illustrations.)  The lattice paths
that we consider are sequences of steps in $\mathbb{R}^2$, each of
unit length, each going north, $N$, or east, $E$.  Fix two lattice
paths $P$ and $Q$ from $(0,0)$ to a point $(m,r)$, where $P$ never
rises above $Q$.  Thus, for each $i$ with $1\leq i\leq r$, if the
$i$th north step of $P$ is in position $b_i$ in $P$, and the $i$th
north step of $Q$ is in position $a_i$ in $Q$, then $a_i\leq b_i$.
The paths $P$ and $Q$ bound a region $\mathcal{R}$ in $\mathbb{R}^2$;
let $\mathcal{P}$ be the set of lattice paths from $(0,0)$ to $(m,r)$
that remain in $\mathcal{R}$.  For $P' \in \mathcal{P}$, viewed as a
word in the alphabet $\{E,N\}$, let $b(P')$ be the set of positions in
$P'$ where $N$ occurs.  Note that the position, in a lattice path, of
any step that ends at $(s,t)$ is $s+t$, so if we put the label $s+t$
on the line segment (a north step) from $(s,t-1)$ to $(s,t)$, then
$b(P')$ is the set of labels on the north steps in the path $P'$.  As
shown in~\cite{lpm1}, the set $\{b(P')\,:\,P'\in\mathcal{P}\}$ is the
set of bases of a transversal matroid, denoted by $M[P,Q]$, and one
presentation of this transversal matroid is given by
$\{\{a_i,a_i+1,\ldots,b_i\}\,:\,1\leq i\leq r\}$; these sets are the
sets of labels on the North steps in a fixed row of the lattice path
diagram.  A lattice path matroid is a matroid that is isomorphic to
some such matroid $M[P,Q]$.

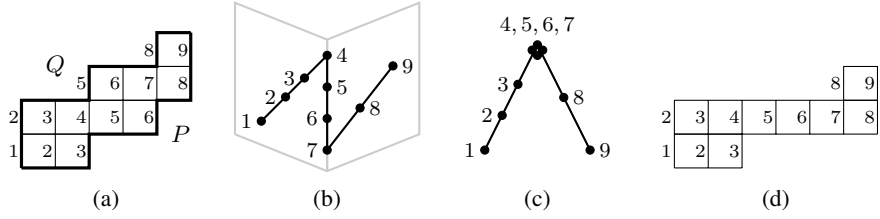
\begin{figure}
  \centering
  \begin{tikzpicture}[scale=0.45]
    \draw  (0,0)  grid  (2, 1);
    \draw  (0,1)  grid  (4, 2);
    \draw  (2,2)  grid  (5, 3);
    \draw  (4,3)  grid  (5, 4);
    \draw[very thick](0,0)--(2,0)--(2,1)--(4,1)--(4,2)--(5,2)
    --(5,2)--(5,4);
    \node at (4.7,1) {\small $P$};
    \draw[very thick](0,0)--(0,2)--(2,2)--(2,3)--(4,3)--(4,4)--(5,4);
    \node at (1,3) {\small $Q$};
    \node at (-0.23,0.5) {\scriptsize $1$};
    \node at (0.77,0.5) {\scriptsize $2$};
    \node at (1.77,0.5) {\scriptsize $3$};
    \node at (-0.23,1.5) {\scriptsize $2$};
    \node at (0.77,1.5) {\scriptsize $3$};
    \node at (1.77,1.5) {\scriptsize $4$};
    \node at (2.77,1.5) {\scriptsize $5$};
    \node at (3.77,1.5) {\scriptsize $6$};
    \node at (1.77,2.5) {\scriptsize $5$};
    \node at (2.77,2.5) {\scriptsize $6$};
    \node at (3.77,2.5) {\scriptsize $7$};
    \node at (4.77,2.5) {\scriptsize $8$};
    \node at (3.77,3.5) {\scriptsize $8$};
    \node at (4.77,3.5) {\scriptsize $9$};
    \node at (2.5,-0.95) {\small (a)};
  \end{tikzpicture}\hspace{10pt}
\begin{tikzpicture}[scale=0.7]
\draw[thick, black!20]
(0,0.5)--(1.75,-0.2)--(3.5,0.5)--(3.5,3)--(1.75,2.3)--(0,3)--(0,0.5);
\draw[thick, black!20] (1.75,-0.2)--(1.75,2.3);
\draw[thick, black!20] (1.75,-0.2)--(1.75,2.3);
\draw[thick](0.5,0.75)--(1.75,2)--(1.75,0.2)--(3,1.8);
\filldraw (0.5,0.75) node[left] {\footnotesize $1$} circle  (2.2pt);
\filldraw (0.96,1.21) node[left] {\footnotesize $2$} circle  (2.2pt);
\filldraw (1.32,1.57) node[left] {\footnotesize $3$} circle  (2.2pt);
\filldraw (1.75,2) node[right] {\footnotesize $4$} circle  (2.2pt);
\filldraw (1.75,1.4) node[right] {\footnotesize $5$} circle  (2.2pt);
\filldraw (1.75,0.8) node[left] {\footnotesize $6$} circle  (2.2pt);
\filldraw (1.75,0.2) node[left] {\footnotesize $7$} circle  (2.2pt);
\filldraw (2.375,1) node[right] {\footnotesize $8$} circle  (2.2pt);
\filldraw (3,1.8) node[right] {\footnotesize $9$} circle  (2.2pt);
    \node at (1.75,-0.75) {\small (b)};
  \end{tikzpicture}\hspace{10pt}
\begin{tikzpicture}[scale=0.7]
\draw[thick](0,0)--(1,2)--(2,0);
\filldraw (0.0,0.0) node[left] {\footnotesize $1$} circle  (2.2pt);
\filldraw (0.33,0.66) node[left] {\footnotesize $2$} circle  (2.2pt);
\filldraw (0.63,1.25) node[left] {\footnotesize $3$} circle  (2.2pt);
\filldraw (1,2) node[above] {\footnotesize $4,5,6,7$} circle  (2.2pt);
\filldraw (0.9,1.9) circle  (2.2pt);
\filldraw (1.1,1.9) circle  (2.2pt);
\filldraw (1,1.8) circle  (2.2pt);
\filldraw (1.5,1) node[right] {\footnotesize $8$} circle  (2.2pt);
\filldraw (2,0) node[right] {\footnotesize $9$} circle  (2.2pt);
    \node at (1,-0.95) {\small (c)};
  \end{tikzpicture}\hspace{10pt}
  \begin{tikzpicture}[scale=0.45]
    \draw  (0,0)  grid  (2, 1);
    \draw  (0,1)  grid  (6, 2);
    \draw  (5,2)  grid  (6, 3);
    \node at (-0.23,0.5) {\scriptsize $1$};
    \node at (0.77,0.5) {\scriptsize $2$};
    \node at (1.77,0.5) {\scriptsize $3$};
    \node at (-0.23,1.5) {\scriptsize $2$};
    \node at (0.77,1.5) {\scriptsize $3$};
    \node at (1.77,1.5) {\scriptsize $4$};
    \node at (2.77,1.5) {\scriptsize $5$};
    \node at (3.77,1.5) {\scriptsize $6$};
    \node at (4.77,1.5) {\scriptsize $7$};
    \node at (5.77,1.5) {\scriptsize $8$};
    \node at (4.77,2.5) {\scriptsize $8$};
    \node at (5.77,2.5) {\scriptsize $9$};
    \node at (3,-0.95) {\small (d)};
  \end{tikzpicture}
  \caption{Examples of (a) the region of interest, (b) the lattice
    path matroid it gives, which is the transversal matroid that has
    the presentation
    $\{\{1,2,3\},\{2,\ldots,6\},\{5,\ldots,8\},\{8,9\}\}$, (c) a
    quotient of that matroid, and (d) a region that yields the
    quotient.}
  \label{fig:lpmexamples}
\end{figure}

To extend this construction to delta-matroids, we take regions that
are bounded by a pair of lattice paths as
Figure~\ref{fig:pathsgeneral} illustrates.  Specifically, we have four
lattice points $s_P = (0,0)$, $s_Q = (-d,d)$, $t_Q = (u,v)$, and
$t_P = (u+c,v-c)$ where $v-c\geq d$ and $u,c,d\geq 0$, and two lattice
paths, $P$ from $s_P$ to $t_P$, and $Q$ from $s_Q$ to $t_Q$, with $P$
never crossing $Q$.  These two paths, the line through $s_P$ and
$s_Q$, and that through $t_P$ and $t_Q$, bound a region in
$\mathbb{R}^2$, which we denote by $\mathcal{R}$.  Label the lattice
points between $s_P$ and $s_Q$ as shown, and do likewise for those
between $t_P$ and $t_Q$.  We label each north step in $\mathcal{R}$
from $1$ to $u+v$ according to the sum of the coordinates of its
higher endpoint, and we let $E$ be the set $\{1,2,\ldots,u+v\}$ of all
such labels.  Let $\mathcal{P}$ be the set of lattice paths from some
$s_i$ to some $t_j$ that remain in $\mathcal{R}$.  With each path
$P'\in \mathcal{P}$, let $b(P')$ be the set of labels on its north
steps.  The set
\[\{b(P')\,:\,P'\in\mathcal{P} \text{ going from } s_Q \text{ to } t_P\}\]
is the set of bases of a lattice path matroid on $E$, which we denote
by $M(\mathcal{R}_{\min})$.  Likewise,
\[\{b(P')\,:\,P'\in\mathcal{P} \text{ going from } s_P \text{ to }
  t_Q\}\] is the set of bases of a lattice path matroid on $E$, which
we denote by $M(\mathcal{R}_{\max})$.  Below we show that
$M(\mathcal{R}_{\min})$ is a quotient of $M(\mathcal{R}_{\max})$ and
that the sets $b(P')$, over all $P' \in \mathcal{R}$, are the feasible
sets of the full Higgs lift delta-matroid for this pair of matroids.
It is not hard to check that there is no region $\mathcal{R}$ for
which $M(\mathcal{R}_{\max})$ and $M(\mathcal{R}_{\min})$ are
isomorphic to the two matroids in Figure~\ref{fig:lpmexamples}.  Thus,
this construction does not yield all quotient-lift pairs of lattice
path matroids.

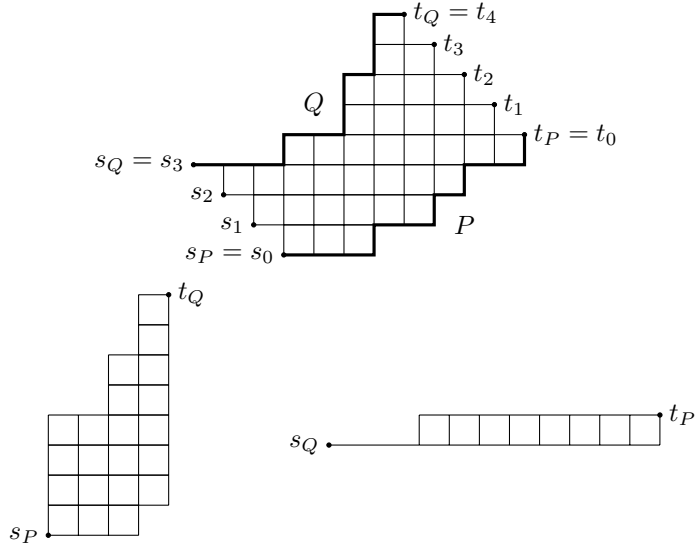
\begin{figure}
  \centering
  \begin{tikzpicture}[scale=0.4]
    \draw  (3,-1)  grid  (6,0);
    \draw  (2,0)  grid  (8, 1);
    \draw  (1,1)  grid  (9, 2);
    \draw  (3,2)  grid  (11, 3);
    \draw  (5,3)  grid  (10, 4);
    \draw  (5,4)  grid  (9, 5);
    \draw  (6,5)  grid  (8, 6);
    \draw  (6,6)  grid  (7, 7);
    \filldraw (11,3) node[right] {$t_P=t_0$} circle  (2pt);
    \filldraw (10,4) node[right] {$t_1$} circle  (2pt);
    \filldraw (9,5) node[right] {$t_2$} circle  (2pt);
    \filldraw (8,6) node[right] {$t_3$} circle  (2pt);
    \filldraw (7,7) node[right] {$t_Q=t_4$} circle  (2pt);
    \filldraw (3,-1) node[left] {$s_P=s_0$} circle  (2pt);
    \filldraw (2,0) node[left] {$s_1$} circle  (2pt);
    \filldraw (1,1) node[left] {$s_2$} circle  (2pt);
    \filldraw (0,2) node[left] {$s_Q=s_3$} circle  (2pt);
    \draw[very thick](3,-1)--(6,-1)--(6,0)--(8,0)--(8,1)--(9,1)--(9,2)--(11,2)--(11,3);
    \node at (9,0) {$P$};
    \draw[very thick] (0,2)--(3,2)--(3,3)--(5,3)--(5,5)--(6,5)--(6,7)--(7,7);
    \node at (4,4) {$Q$};
  \end{tikzpicture}
\\
  \begin{tikzpicture}[scale=0.4]
    \draw  (3,-1)  grid  (6,0);
    \draw  (3,0)  grid  (7, 1);
    \draw  (3,1)  grid  (7, 2);
    \draw  (3,2)  grid  (7, 3);
    \draw  (5,3)  grid  (7, 4);
    \draw  (5,4)  grid  (7, 5);
    \draw  (6,5)  grid  (7, 6);
    \draw  (6,6)  grid  (7, 7);
    \filldraw (7,7) node[right] {$t_Q$} circle  (2pt);
    \filldraw (3,-1) node[left] {$s_P$} circle  (2pt);
  \end{tikzpicture}
\hspace{20pt}
  \begin{tikzpicture}[scale=0.4]
    \draw  (3,2)  grid  (11, 3);
    \filldraw[white] (3,-1) node[left] {$s_P$} circle  (0pt);
    \filldraw (11,3) node[right] {$t_P$} circle  (2pt);
    \filldraw (0,2) node[left] {$s_Q$} circle  (2pt);
    \draw(0,2)--(3,2);
  \end{tikzpicture}
  \caption{Above, a typical region of interest.  Below, the lattice
    path representations of the two associated lattice path matroids,
    $M(\mathcal{R}_{\max})$ and $M(\mathcal{R}_{\min})$.}
  \label{fig:pathsgeneral}
\end{figure}

\begin{prop}\label{prop:genlpquot}
  With the notation above,
  \begin{enumerate}
  \item $M(\mathcal{R}_{\min})$ is a quotient of $M(\mathcal{R}_{\max})$, and
  \item the map $P'\mapsto b(P')$ is a surjection from $\mathcal{P}$
    onto the set of feasible sets of the full
    Higgs lift delta-matroid of
    the pair $(M(\mathcal{R}_{\min}), M(\mathcal{R}_{\max}))$.
  \end{enumerate}
\end{prop}

\begin{proof}
  Let $B$ be a basis of $M(\mathcal{R}_{\max})$.  Fix $e$ in $E-B$.
  We will verify the condition in part~(\ref{lemlabel:quobasis}) of
  Lemma~\ref{lem:quotviabases}.  View $B$ as a lattice path, say
  $B=b(P_B)$.  To get the required basis $B'$ of
  $M(\mathcal{R}_{\min})$ (viewed as a lattice path, $P_{B'}$), take
  east steps from $s_Q$ until $P_B$ is reached, then follow $P_B$
  until a final sequence of east steps goes directly to $t_P$.  (See
  Figure~\ref{fig:pathsgeneralsketch}.)  Assume that $f\in B'$ and
  $(B'\cup e)-f$ is a basis of $M(\mathcal{R}_{\min})$.  Note that
  paths $P_B$ and $P_{B'}$ share step $f$.
  Figure~\ref{fig:pathsgeneralswap} compares the paths that correspond
  to $B'$ and $(B'\cup e)-f$.  It follows that if $P_B$ and $P_{B'}$
  share step $e$, then since the path corresponding to $(B'\cup e)-f$
  stays in $\mathcal{R}$, and between steps $e$ and $f$ the paths that
  correspond to $(B'\cup e)-f$ and $(B\cup e)-f$ are identical, we
  have $(B\cup e)-f\in \mathcal{P}$.  If $P_B$ and $P_{B'}$ do not
  share step $e$, then we may assume by symmetry that $e$ is after the
  last step that $P_B$ and $P_{B'}$ share.  In this case the
  modifications of $P_B$ and $P_{B'}$ to get the paths for
  $(B\cup e)-f$ and $(B'\cup e)-f$ differ just in the sort of regions
  that are shaded with hatch lines in
  Figure~\ref{fig:pathsgeneralsketch}, which are in $\mathcal{R}$.
  Thus, these paths stay in $\mathcal{R}$, so
  $(B\cup e)-f\in \mathcal{P}$ and assertion (1) holds.

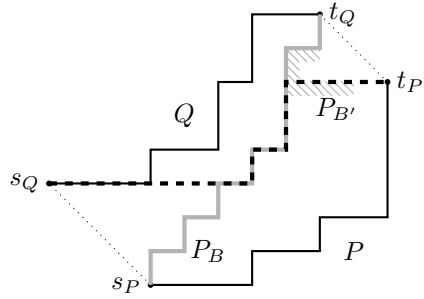
\begin{figure}
  \centering
  \begin{tikzpicture}[scale=0.45]
    \draw[dotted] (3,-1)--(0,2);
    \draw[dotted] (8,7)--(10,5);
\fill[pattern = north west lines, pattern color= black!30]
 (7,4.6)--(9,4.6)--(9,5)--(7,5)--(7,4.6);
\fill[pattern = north west lines, pattern color= black!30]
 (7,5)--(7.4,5)--(7.4,5.6)--(8,5.6)--(8,6)--(7,6)--(7,5);
    \filldraw (10,5) node[right] {$t_P$} circle  (2pt);
     \filldraw (8,7) node[right] {$t_Q$} circle  (2pt);
    \filldraw (3,-1) node[left] {$s_P$} circle  (2pt);
     \filldraw (0,2) node[left] {$s_Q$} circle  (2pt);
    \draw[thick](3,-1)--(6,-1)--(6,0)--(8,0)--(8,1)--(10,1)--(10,5);
    \node at (9,0) {$P$};
    \draw[thick] (0,2)--(3,2)--(3,3)--(5,3)--(5,5)--(6,5)--(6,7)--(8,7);
    \node at (4,4) {$Q$};
    \draw[ultra thick,black!30]
    (3,-1)--(3,0)--(4,0)--(4,1)--(5,1)--(5,2)--(6,2)--(6,3)--(7,3)
    --(7,6)--(8,6)--(8,7);
    \node at (4.7,0) {$P_B$};
    \draw[ultra thick,dashed]
    (0,2)--(6,2)--(6,3)--(7,3)--(7,5)--(10,5);
    \node at (8.5,4.2) {$P_{B'}$};
  \end{tikzpicture}
  \caption{A sketch of how to get the path $P_{B'}$ (dashed) from
    $P_B$ (in gray) in the proof of Proposition~\ref{prop:genlpquot}.}
  \label{fig:pathsgeneralsketch}
\end{figure}

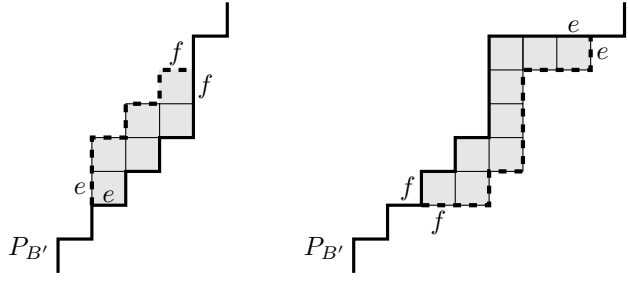
\begin{figure}
  \centering
  \begin{tikzpicture}[scale=0.45]
  \filldraw[color= black!10]
  (4,1) -- (4,3) -- (5,3) -- (5,4)--(6,4)--(6,5)--(7,5)--(7,3)--
  (6,3)--(6,2) --(5,2)--(5,1)--(4,1);
    \draw  (4,1)  grid  (5,3);
    \draw  (5,3)  grid  (7,4);
  \draw[very thick](3,-1)--(3,0)--(4,0)--(4,1)--(5,1)--(5,2)--(6,2)--(6,3)--(7,3)
    --(7,6)--(8,6)--(8,7);
  \draw[ultra thick, dashed](4,1)--(4,3)--(5,3)--(5,4)--(6,4)--(6,5)--(7,5);
  \node at (2.15,-0.25) {$P_{B'}$};
  \node at (7.35,4.5) {$f$};
  \node at (6.5,5.5) {$f$};
  \node at (4.5,1.25) {$e$};
  \node at (3.65,1.5) {$e$};
  \end{tikzpicture}
\hspace{20pt}
  \begin{tikzpicture}[scale=0.45]
  \filldraw[color= black!10]
  (5,1) -- (7,1) -- (7,2) -- (8,2)--(8,5)--(10,5)--(10,6)--(7,6)--(7,3)--
  (6,3)--(6,2) --(5,2)--(5,1);
    \draw  (8,6)  grid  (7,2);
    \draw  (8,6)  grid  (10,5);
    \draw  (5,1)  grid  (7,2);
  \draw[very thick](3,-1)--(3,0)--(4,0)--(4,1)--(5,1)--(5,2)--(6,2)--(6,3)--(7,3)
    --(7,6)--(11,6)--(11,7);
  \draw[ultra thick, dashed](5,1)--(7,1)--(7,2)--(8,2)--(8,5)--(10,5)--(10,6);
  \node at (2.15,-0.25) {$P_{B'}$};
  \node at (10.35,5.5) {$e$};
  \node at (9.5,6.3) {$e$};
  \node at (4.6,1.5) {$f$};
  \node at (5.5,0.5) {$f$};
  \end{tikzpicture}
  \caption{Exchanging $f$ for a smaller element $e$ diverts the solid
    path around the shaded region to the left, as the dashed path in
    the first part shows.  Exchanging $f$ for a larger element $e$
    diverts the path around the shaded region to the right, as the
    dashed path in the second part shows.
  }
  \label{fig:pathsgeneralswap}
\end{figure}

For part (2), consider a path $P'\in\mathcal{P}$, say from $s_u$ to
$t_v$, as in Figure~\ref{fig:pathsforHiggls}.  A subpath of $P'$ goes
from a point with the same $y$-coordinate as $s_Q$ to one with the
same $y$-coordinate as $t_P$, and the set of labels on the north steps
in that subpath is clearly a basis of $M(\mathcal{R}_{\min})$.
Figure~\ref{fig:pathsforHiggls} shows how to create a path $P''$ from
$s_P$ to $t_Q$ with $b(P')\subseteq b(P'')$.  Thus, for each path $P'$
in $\mathcal{P}$, the set $b(P')$ is a basis of a Higgs lift of
$M(\mathcal{R}_{\min})$ to $M(\mathcal{R}_{\max})$.

We turn to the converse, showing that each basis $B$ of each Higgs
lift of $M(\mathcal{R}_{\min})$ to $M(\mathcal{R}_{\max})$ is $b(P')$
for some $P'\in\mathcal{P}$, that is, if $B_0$ is a basis of
$M(\mathcal{R}_{\min})$ and $B_1$ is a basis of
$M(\mathcal{R}_{\max})$, and if $B_0\subseteq B\subseteq B_1$, then
$B=b(P')$ for some path $P'\in\mathcal{P}$.  We induct on $|B_1-B|$.
The base case, $B=B_1$, is obvious, so assume that $|B_1-B|>0$ and
that the assertion holds for all diagrams $\mathcal{R}'$ and triples
$B'_0\subseteq B'\subseteq B'_1$ where $B'_0$ is a basis of
$M(\mathcal{R}'_{\min})$ and $B'_1$ is a basis of
$M(\mathcal{R}'_{\max})$ and $|B'_1-B'|<|B_1-B|$.

Let $I_1$ be the interval of labels on the lowest row of north steps
in $\mathcal{R}$, and likewise for successive rows.  We call an
interval $I_j$ lower, middle, or upper according to whether the
corresponding row is below $s_Q$, between $s_Q$ and $t_P$, or above
$t_P$.  Let $P_{B_1}$ be the path with $b(P_{B_1})=B_1$.  We call an
interval good if the north step that $P_{B_1}$ uses in it is in $B$;
otherwise it is bad.  Since $|B_1-B|>0$, there is at least one bad
interval.

\begin{figure}
  \centering
  \begin{tikzpicture}[scale=0.45]
    \draw[dotted] (3,-1)--(0,2);
    \draw[dotted] (8,7)--(10,5);
    \filldraw (10,5) node[right] {$t_P$} circle  (2pt);
     \filldraw (8,7) node[right] {$t_Q$} circle  (2pt);
    \filldraw (3,-1) node[left] {$s_P$} circle  (2pt);
     \filldraw (0,2) node[left] {$s_Q$} circle  (2pt);
    \draw[thick](3,-1)--(6,-1)--(6,0)--(8,0)--(8,1)--(10,1)--(10,5);
    \node at (9,0) {$P$};
    \node at (6.8,2.5) {$P'$};
    \draw[thick] (0,2)--(3,2)--(3,3)--(5,3)--(5,5)--(6,5)--(6,7)--(8,7);
    \node at (4,4) {$Q$};
    \draw[ultra thick, dashed] (4,2)--(0,2);
    \draw[ultra thick, dashed] (9,5)--(10,5);
    \draw[very thick, dotted] (3,-1)--(3,1);
    \draw[very thick, dotted] (8,7)--(8,4);
    \draw[very thick,black!30]
    (2,0)--(2,1)--(4,1)--(4,2)--(6,2)--(6,3)--(7,3)--(7,4)--(9,4)--
   (9,5)--(9,6);
  \end{tikzpicture}
  \caption{The gray line is a path $P'$ from $s_u$ to $t_v$.  The
    dashed lines show that $b(P')$ contains a basis of
    $M(\mathcal{R}_{\min})$.  The dotted lines show that $b(P')$ is
    contained in a basis of $M(\mathcal{R}_{\max})$.}
  \label{fig:pathsforHiggls}
\end{figure}
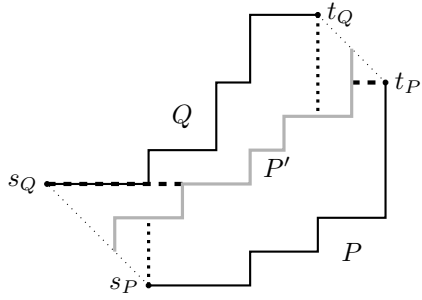

First assume that there is a bad lower interval, say $I_h$.  Let the
north step that $P_{B_1}$ uses in $I_h$ be labeled $x$, so
$x\in B_1-B$.  Each lower interval properly contains those below it,
so if we delete interval $I_1$ from the diagram (adjusting $P$ and
$s_P$ accordingly) to get a region $\mathcal{R}'$, then $B_1-x$ is a
basis of $M(\mathcal{R}'_{\max})$ and the induction hypothesis applies
to $\mathcal{R}'$, $B$, and $B_1-x$ since $|(B_1-x)-B|<|B_1-B|$. (The
path that corresponds to $B_1-x$ is obtained from $P_{B_1}$ by moving
each step before $x$ northwest and changing $x$ to an east step, as
shown in Figure~\ref{fig:pathsbadlower}, so the path remains in
$\mathcal{R}'$).  By induction there is a path $P'$ in $\mathcal{R}'$
with $b(P')=B$, and since $\mathcal{R}$ contains $\mathcal{R}'$, this
path $P'$ is also a path in $\mathcal{R}$, as we needed.

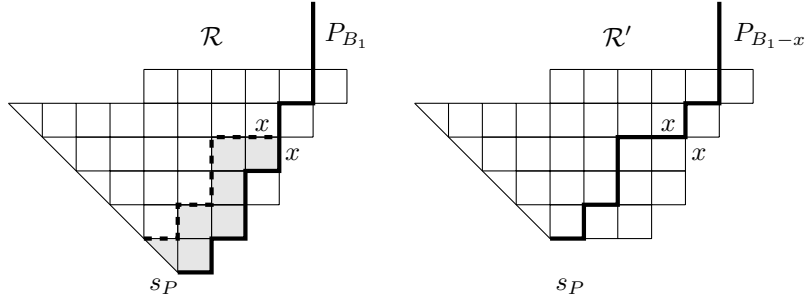
\begin{figure}
  \centering
  \begin{tikzpicture}[scale=0.45]
  \filldraw[color= black!10]
  (0,0) -- (-1,1) -- (0,1) -- (0,2) -- (1,2) -- (1,4)--(3,4)--(3,3)--
  (2,3)--(2,1) --(1,1)--(1,0)--(0,0);
    \draw  (0,0)  grid  (1,2);
    \draw  (1,2)  grid  (2,4);
    \draw  (-1,1)  grid  (2,2);
    \draw  (-2,2)  grid  (3,3);
    \draw  (-3,3)  grid  (3,4);
    \draw  (-4,4)  grid  (4,5);
        \draw  (-1,5)  grid  (5,6);
    \draw (-5,5)--(-4,5);
  \draw[ultra thick] (0,0)--(1,0)--(1,1)--(2,1)--(2,3)--(3,3)
    --(3,5)--(4,5)--(4,8);
  \draw[ultra thick, dashed] (-1,1)--(0,1)--(0,2)--(1,2)--(1,4)--(3,4);
  \draw (0,0)--(-5,5);
    \node at (1,7) {$\mathcal{R}$};
  \node at (5,7) {$P_{B_1}$};
  \node at (3.4,3.5) {$x$};
  \node at (2.5,4.35) {$x$};
  \node at (-0.4,-0.4) {$s_P$};
    \draw  (11,1)  grid  (14,2);
    \draw  (10,2)  grid  (15,3);
    \draw  (9,3)  grid  (15,4);
    \draw  (8,4)  grid  (16,5);
    \draw  (11,5)  grid  (17,6);
    \draw (7,5)--(8,5);
    \draw[ultra thick] (11,1)--(12,1)--(12,2)--(13,2)--(13,4)--(15,4)
    --(15,5)--(16,5)--(16,8);
  \draw (11,1)--(7,5);
  \node at (13,7) {$\mathcal{R}'$};
  \node at (17.5,7) {$P_{B_1-x}$};
  \node at (15.4,3.5) {$x$};
  \node at (14.5,4.35) {$x$};
  \node at (11.6,-0.4) {$s_P$};
  \end{tikzpicture}
  \caption{To treat a bad lower interval, replace the path $P_{B_1}$,
    shown in bold on the left, with $P_{B_1-x}$.  Only the lower rows
    of the diagrams are shown.}
  \label{fig:pathsbadlower}
\end{figure}

We can treat bad upper intervals similarly (deleting the top
interval), so now assume that the only bad intervals are middle
intervals.  When there are at least two bad middle intervals, we
choose which to process as follows.  Let $I_j$ and $I_k$ be the lowest
and highest such intervals, respectively.  Let $P_{B_0}$ be the path
with $b(P_{B_0})=B_0$.  Let the north step that $P_{B_1}$ uses in
$I_j$ be $x$, so $x\in B_1-B$, and let the north step that $P_{B_0}$
uses in $I_j$ be $y$, so $y\in B_0$, so $y\ne x$.  Let $x'$ and $y'$
be the elements of $B_1-B$ and $B_0$, respectively, defined in the
same way using $I_k$.  We cannot have both $x<y$ and $y'<x'$ since
$B_0\subseteq B_1$ and since $P_{B_1}$ and $P_{B_0}$ use exactly one
north step from each of $I_j,I_{j+1},\ldots,I_k$.  Now assume $y<x$.
(The case of $x'<y'$ is handled similarly, working with the intervals
above $I_k$.)  Let $I_h$ be the lowest middle interval, and let
$x_{h-1}<x_h<\cdots<x_j=x$ be the elements of $B_1$ that $P_{B_1}$
uses as north steps in $I_{h-1},I_h,\ldots,I_j$.  Likewise, let
$y_h<y_{h+1}<\cdots<y_j=y$ be the elements of $B_0$ that $P_{B_0}$
uses as north steps in $I_h,I_{h+1},\ldots,I_j$.  Since
$B_0\subseteq B_1$, from $y_j<x_j$, we get $y_i\leq x_{i-1}<x_i$ for
all $i$ with $h\leq i\leq j$; thus, $x_{i-1}\in I_i$.  From this, it
is easy to see that if we delete the interval $I_1$ from the diagram
to get a region $\mathcal{R}'$, then, as in the case we treated above,
the induction hypothesis applies to $\mathcal{R}'$, $B$, and $B_1-x$,
and yields the path $P'$ in $\mathcal{R}$ that we needed.
\end{proof}

We call the delta-matroids constructed above, and delta-matroids that
are isomorphic to them, \emph{lattice path delta-matroids}.

\begin{prop}
  The class of lattice path delta-matroids is closed under duals and
  minors.
\end{prop}

\begin{figure}
  \centering
  \begin{tikzpicture}[scale=0.4]
    \draw  (3,-1)  grid  (4,0);
    \draw  (2,0)  grid  (4, 1);
    \draw  (1,1)  grid  (4, 2);
    \draw  (3,2)  grid  (4, 3);
    \draw  (5,3)  grid  (10, 4);
    \draw  (5,4)  grid  (9, 5);
    \draw  (6,5)  grid  (8, 6);
    \draw  (6,6)  grid  (7, 7);
    \filldraw (11,3) node[right] {$t_P$} circle  (2pt);
    \filldraw (10,4) node[right] {$t_1$} circle  (2pt);
    \filldraw (9,5) node[right] {$t_2$} circle  (2pt);
    \filldraw (8,6) node[right] {$t_3$} circle  (2pt);
    \filldraw (7,7) node[right] {$t_Q$} circle  (2pt);
    \filldraw (3,-1) node[left] {$s_P$} circle  (2pt);
    \filldraw (2,0) node[left] {$s_1$} circle  (2pt);
    \filldraw (1,1) node[left] {$s_2$} circle  (2pt);
    \filldraw (0,2) node[left] {$s_Q$} circle  (2pt);
    \draw[thick](3,-1)--(4,-1)--(4,3)--(11,3);
    \node at (7,2) {$P$};
    \draw[thick] (0,2)--(3,2)--(3,3)--(5,3)--(5,5)--(6,5)--(6,7)--(7,7);
    \node at (4,4) {$Q$};
    \draw[ultra thick,black!20](4,3)--(5,3);
    \node at (4.5,2.5) {$e$};
  \end{tikzpicture}
\hspace{20pt}
  \begin{tikzpicture}[scale=0.4]
    \draw  (3,-1)  grid  (4,0);
    \draw  (2,0)  grid  (4, 1);
    \draw  (1,1)  grid  (4, 2);
    \draw  (3,2)  grid  (4, 3);
    \draw  (4,3)  grid  (9, 4);
    \draw  (4,4)  grid  (8, 5);
    \draw  (5,5)  grid  (7, 6);
    \draw  (5,6)  grid  (6, 7);
    \filldraw (10,3) node[right] {$t_P$} circle  (2pt);
    \filldraw (9,4) node[right] {$t_1$} circle  (2pt);
    \filldraw (8,5) node[right] {$t_2$} circle  (2pt);
    \filldraw (7,6) node[right] {$t_3$} circle  (2pt);
    \filldraw (6,7) node[right] {$t_Q$} circle  (2pt);
    \filldraw (3,-1) node[left] {$s_P$} circle  (2pt);
    \filldraw (2,0) node[left] {$s_1$} circle  (2pt);
    \filldraw (1,1) node[left] {$s_2$} circle  (2pt);
    \filldraw (0,2) node[left] {$s_Q$} circle  (2pt);
    \draw[thick](3,-1)--(4,-1)--(4,3)--(10,3);
    \node at (7,2) {$P'$};
    \draw[thick] (0,2)--(3,2)--(3,3)--(4,3)--(4,5)--(5,5)--(5,7)--(6,7);
    \node at (3,4) {$Q'$};
  \end{tikzpicture}
  \caption{Deleting or contracting a loop, $e$ (in gray).}
  \label{fig:lpmminorloop}
\end{figure}
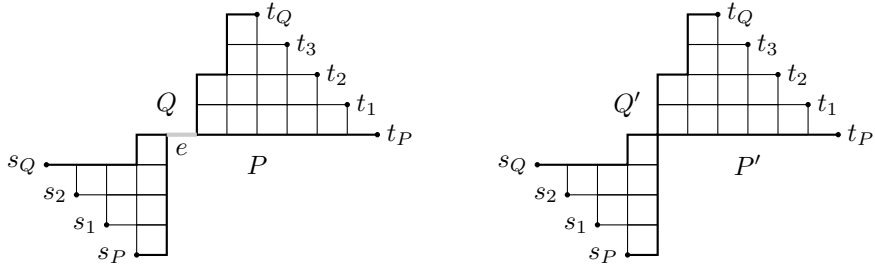

\begin{figure}
  \centering
  \begin{tikzpicture}[scale=0.4]
    \draw  (3,-1)  grid  (6,0);
    \draw  (2,0)  grid  (8, 1);
    \draw  (1,1)  grid  (9, 2);
    \draw  (3,2)  grid  (11, 3);
    \draw  (5,3)  grid  (10, 4);
    \draw  (5,4)  grid  (9, 5);
    \draw  (6,5)  grid  (8, 6);
    \draw  (6,6)  grid  (7, 7);
    \filldraw (11,3) node[right] {$t_P$} circle  (2pt);
    \filldraw (10,4) node[right] {$t_1$} circle  (2pt);
    \filldraw (9,5) node[right] {$t_2$} circle  (2pt);
    \filldraw (8,6) node[right] {$t_3$} circle  (2pt);
    \filldraw (7,7) node[right] {$t_Q$} circle  (2pt);
    \filldraw (3,-1) node[left] {$s_P$} circle  (2pt);
    \filldraw (2,0) node[left] {$s_1$} circle  (2pt);
    \filldraw (1,1) node[left] {$s_2$} circle  (2pt);
    \filldraw (0,2) node[left] {$s_Q$} circle  (2pt);
    \draw[thick](3,-1)--(6,-1)--(6,0)--(8,0)--(8,1)--(9,1)--(9,2)--(11,2)--(11,3);
    \node at (9,0) {$P$};
    \draw[thick] (0,2)--(3,2)--(3,3)--(5,3)--(5,5)--(6,5)--(6,7)--(7,7);
    \node at (4,4) {$Q$};
    \draw[ultra thick,black!20]
    (5,5)--(5,4)--(6,4)--(6,3)--(7,3)--(7,2)--(8,2)--(8,1)--(9,1);
    \draw[very thick,dotted]
    (5,5)--(5,4)--(6,4)--(6,3)--(7,3)--(7,2)--(8,2)--(8,1)--(9,1);
  \end{tikzpicture}
\hspace{10pt}
  \begin{tikzpicture}[scale=0.4]
    \draw  (3,-1)  grid  (6,0);
    \draw  (2,0)  grid  (8, 1);
    \draw  (1,1)  grid  (9, 2);
    \draw  (3,2)  grid  (11, 3);
    \draw  (5,3)  grid  (10, 4);
    \draw  (5,4)  grid  (9, 5);
    \draw  (6,5)  grid  (8, 6);
    \draw  (6,6)  grid  (7, 7);
    \filldraw (11,3) node[right] {$t_P$} circle  (2pt);
    \filldraw (10,4) node[right] {$t_1$} circle  (2pt);
    \filldraw (9,5) node[right] {$t_2$} circle  (2pt);
    \filldraw (8,6) node[right] {$t_3$} circle  (2pt);
    \filldraw (7,7) node[right] {$t_Q$} circle  (2pt);
    \filldraw (3,-1) node[left] {$s_P$} circle  (2pt);
    \filldraw (2,0) node[left] {$s_1$} circle  (2pt);
    \filldraw (1,1) node[left] {$s_2$} circle  (2pt);
    \filldraw (0,2) node[left] {$s_Q$} circle  (2pt);
    \draw[thick](3,-1)--(6,-1)--(6,0)--(8,0)--(8,1)--(9,1)--(9,2)--(11,2)--(11,3);
    \node at (9,0) {$P$};
    \draw[thick] (0,2)--(3,2)--(3,3)--(5,3)--(5,5)--(6,5)--(6,7)--(7,7);
    \node at (4,4) {$Q$};
    \draw[ultra thick,black!20]
    (5,5)--(5,4)--(6,4)--(6,3)--(7,3)--(7,2)--(8,2)--(8,1)--(9,1);
   \draw[very thick,dotted]
    (5,5)--(5,4)--(6,4)--(6,3)--(7,3)--(7,2)--(8,2)--(8,1)--(9,1);
    \draw[ultra thick,white](5,5)--(5,4);
    \draw[ultra thick,white](6,4)--(6,3);
    \draw[ultra thick,white](7,3)--(7,2);
    \draw[ultra thick,white](8,2)--(8,1);
    \draw[ultra thick,white](5,5)--(6,5);
   \fill[pattern = north west lines] (4.9,4.75) rectangle (6,5.25);
    \draw[ultra thick](5,5)--(6,5);
  \end{tikzpicture}
  \caption{The dotted line shows the steps that are labelled $e$.}
  \label{fig:lpmminor}
\end{figure}
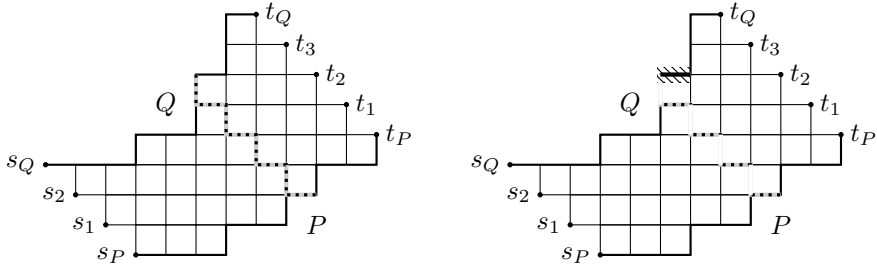

\begin{proof}
  Dual-closure is seen by flipping the diagram around the line $y=x$.
  For minors, first note that a loop in a lattice path delta-matroid
  is represented by an east step that is in both bounding paths (thus
  pinching the paths together for at least that step and giving a
  direct sum decomposition).  The deletion and contraction of a loop
  is obtained by eliminating this step and moving the right side of
  the diagram one unit to the left, as Figure~\ref{fig:lpmminorloop}
  illustrates.  The identification and treatment of coloops follows by
  duality.  Now assume that $e$ is neither a loop nor a coloop, so $e$
  is represented by both north and east steps, indeed, by all of the
  north and east steps that are at distance $e$ from the initial
  steps, as Figure~\ref{fig:lpmminor} shows.  To delete $e$, we must
  use only such steps that go east, so erase those that go north, as
  the second part of Figure~\ref{fig:lpmminor} shows.  As shown there
  (highlighted with hatch lines), some steps may no longer be reached;
  erase them.  Now shrink the east steps labelled $e$ to points to
  obtain a lattice path representation of the deletion of $e$.
  Contractions are handled dually.
\end{proof}

The matroid $M=(E,\mathcal{F})$ on $E=\{1,2,3,4\}$ in which
$\mathcal{F}$ consists of all two-element subsets of $E$ except
$\{3,4\}$ is a lattice path matroid.  The two-element feasible sets of
the partial dual $M*\{2,3\}$ are $\{1,2\}$, $\{1,3\}$, and $\{3,4\}$,
which are not the bases of a matroid.  Thus, neither the class of
lattice path delta-matroids nor the class of Higgs lift delta-matroids
is closed under partial duals.

With Proposition~\ref{new31}, we can strengthen
Proposition~\ref{prop:genlpquot} in the following way.
\begin{cor}
  With the notation above, let $j=r(M(\mathcal{R}_{\min}))$ and
  $k=r(M(\mathcal{R}_{\max}))$.  Fix a subset $K$ of
  $\{j,j+1,\ldots,k\}$ for which $\{j,j+1,\ldots,k\}-K$ contains no
  pair of consecutive integers.  Then
  $\{b(P) \,:\,P\in\mathcal{P} \text{ and } |b(P)|\in K\}$ is the set
  of feasible sets of a delta-matroid.
\end{cor}

We note that while $M(\mathcal{R}_{\min})$ and $M(\mathcal{R}_{\max})$
are lattice path matroids, the other Higgs lifts of
$M(\mathcal{R}_{\min})$ toward $M(\mathcal{R}_{\max})$ might not be;
they are in the larger class of multi-path matroids~\cite{omer}.

\section{The excluded-minor characterization of
  delta-matroids}\label{sec:exmin}

Delta-matroids form a minor-closed class of set systems. In this
section, we determine the excluded minors that characterize this
minor-closed class. We also prove Theorem~\ref{extraexhiggs}.

The following set systems play many roles in the rest of this
paper. Let
\[S_i=(\{e_1,e_2,\dots ,e_i\},\{\emptyset ,\{e_1,e_2,\dots ,e_i\}\}).\]
Let $\mathcal{S}$ be the set of all twists of the set systems in
$\{S_3,S_4,\dots \}$.  Let
\begin{itemize}
\item $T_1=(\{a,b,c\},\{\emptyset ,\{a,b\},\{a,b,c\}\})$;
\item $T_2=(\{a,b,c\},\{\emptyset ,\{a,b\},\{a,c\},\{a,b,c\}\})$;
\item $T_3=(\{a,b,c\},\{\emptyset ,\{a\},\{a,b\},\{a,b,c\}\})$;
\item $T_4=(\{a,b,c\},\{\emptyset,\{a\},\{a,b\},\{a,c\},\{a,b,c\}\})$;
\item $T_5=(\{a,b,c,d\},\{\emptyset ,\{a,b\},\{a,b,c,d\}\})$;
\item $T_6=(\{a,b,c,d\},\{\emptyset ,\{a,b\},\{a,c\},\{a,b,c,d\}\})$;
\item $T_7=(\{a,b,c,d\},\{\emptyset,\{a,b\},\{a,c\},\{a,d\},
  \{a,b,c,d\}\})$;
\item $T_8=(\{a,b,c,d\},\{\emptyset,\{a\},\{a,b\},\{a,c\},
  \{a,d\},\{a,b,c,d\}\})$.
\end{itemize}
Let $\mathcal{T}$ be the set of all twists of the set systems in
$\{T_1,T_2, \dots ,T_8\}$.
It is easy to check that none of the set systems just defined is a
delta-matroid, so none of the set systems in
$\mathcal{S}\cup \mathcal{T}$ is a delta-matroid.

We first prove Theorem~\ref{extraexhiggs}.  For that, it is useful to
note that, up to isomorphism, there is only one four-element even set
system $S=(E,\mathcal{F})$ such that $\emptyset ,E\in\mathcal{F}$ and
$S$ is not among $U_3 , U_4 , \ldots , U_7, T_5, T_6, T_7, T_7^*$, and
$S_4$.  Its feasible sets are all sets of even cardinality, that is,
$S$ is the even Higgs lift delta-matroid of the pair
$((E,\{\emptyset\}),(E,\{E\}))$.  Note that this proof does not use
the fact that the class of Higgs lift delta-matroids is minor-closed;
that is, instead, a corollary of the proof.

\setcounter{section}{3}
\setcounter{thm}{4}

\begin{proof}[Proof of Theorem~\ref{extraexhiggs}.]
  Suppose first that a delta-matroid $D=(E,\mathcal{F})$ has a minor
  $D'$ that is isomorphic to one of $U_1,U_2,\dots ,U_7$.  Using
  Lemma~\ref{minorisminor} and relabeling, we may assume that
  $D'=U_i=D\ba X/Y$, and its collection of feasible sets is
  \[\{F-Y:F\in\mathcal{F}\text{ and }Y\subseteq F\subseteq E-X\}.\]
  Since $\emptyset$ and $E(U_i)$ are in $\mathcal{F}(U_i)$, the sets
  $Y$ and $E-X$ are in $\mathcal{F}$, so by
  Lemma~\ref{lem:containment} there are sets
  $A\in\mathcal{B}(D_{\min})$ and $B\in\mathcal{B}(D_{\max})$ with
  $A\subseteq Y$ and $E-X\subseteq B$.  The delta-matroid $U_i$ has
  sets $C$ and $C'$ where $|C|=|C'|$, yet only one of $C$ and $C'$ is
  feasible.  It follows that $A\subseteq Y\cup C\subseteq B$ and
  $A\subseteq Y\cup C'\subseteq B$, yet only one of $Y\cup C$ and
  $Y\cup C'$ is feasible in $D$, so $D$ is not a Higgs lift
  delta-matroid by Lemma~\ref{higgsfeas}, as we needed to prove.

  For the remainder of the proof, we will assume that a delta-matroid
  $D=(E,\mathcal{F})$ is not a Higgs lift delta-matroid, and, toward
  deriving a contradiction, that $D$ does not contain any minor
  isomorphic to a member of $\{U_1,U_2,\dots ,U_7\}$.  From
  Corollary~\ref{cor32}, we know that $\mathcal{F}$ is a subset of the
  full Higgs lift delta-matroid of the pair $(D_{\min},D_{\max})$.
  Evidently $D$ is missing some sets whose addition would give a Higgs
  lift delta-matroid.

  For all non-negative integers $i\leq r(D_{\max})-r(D_{\min})$, let
  $N_i$ be the set system
  \[N_i=(E,\{F\in\mathcal{F} : |F|=i+r(D_{\min})\})\] and let
  $H^i=H^i_{D_{\min},D_{\max}}$.  Since $D$ is not a Higgs lift
  delta-matroid, there is some proper set system $N_k$ such that
  $N_k\neq H^k$.  Let $k$ be least with this property.  Thus,
  $0<k<r(D_{\max})-r(D_{\min})$.  From any basis of the matroid $H^k$
  we can obtain any other basis of $H^k$ by a sequence of
  single-element exchanges. Also, all feasible sets in $N_k$ are bases
  of $H^k$ but not conversely.  It follows that there are sets
  $Y\in \mathcal{F}(N_k)$ and $X\in \mathcal{B}(H^k)-\mathcal{F}(N_k)$
  with $|X\triangle Y|=2$.  Let $X\btu Y=\{x,y\}$, where $X=A\cup x$
  and $Y=A\cup y$.  Since $X,Y\in\mathcal{B}(H^k)$,
  Lemma~\ref{higgsfeas} implies that both are spanning in $D_{\min}$
  and independent in $D_{\max}$.  Furthermore, there exist sets
  $F_x\in\mathcal{B}(D_{\min})$ and $G_x\in\mathcal{B}(D_{\max})$ such
  that $F_x\subseteq X\subseteq G_x$.

  We show that
  \begin{sublemma}\label{top}
    \begin{enumerate}
    \item $A\cup\{x,y,z\}\in\mathcal{F}$ for some element $z\in E-A$,
      where $z$ may be $x$; and
    \item $A\in\mathcal{F}$ or $A-a\in\mathcal{F}$ for some element
      $a\in A$.
    \end{enumerate}
  \end{sublemma}
  By applying Axiom~(SE) to $(A\cup y,G_x,x)$, we find that
  $A\cup \{x,y\},A\cup x$, or $A\cup\{x,y,z\}$ is in $\mathcal{F}$ for
  some element $z\in E-A$.  Since $A\cup x\notin\mathcal{F}$, part (1)
  follows.  By applying Axiom~(SE) to $(A\cup y,F_x,y)$, we find that
  $A, A\cup x$, or $A-a$ is in $\mathcal{F}$, for some element
  $a\in A$.  Since $A\cup x\notin\mathcal{F}$, part (2) follows.

  Next, we show that
  \begin{sublemma}\label{notA}
    $A\notin\mathcal{F}$.
  \end{sublemma}
  Suppose $A\in\mathcal{F}$.  Since $A\cup x\notin\mathcal{F}$ and
  $A\cup y\in\mathcal{F}$, and $(D/A)|\{x,y\}$ is not isomorphic to
  $U_1$, we know that $A\cup\{x,y\}\notin\mathcal{F}$.
  By~\ref{top}(1), $A\cup\{x,y,z\}\in\mathcal{F}$ for some
  $z\in E-(A\cup \{x,y\})$.  Let
  $D'=(E',\mathcal{F}')=(D/A)|\{x,y,z\}$.  Then $\mathcal{F}'$
  contains $\emptyset,\{y\}$, and $\{x,y,z\}$, and avoids $\{x\}$ and
  $\{x,y\}$.  Since $D'/y$ is not isomorphic to $U_1$,
  $\{y,z\}\notin\mathcal{F}'$.  By Axiom~(SE) applied to
  $(\emptyset,\{x,y,z\},x)$, we find that $\{x\},\{x,y\}$, or
  $\{x,z\}$ is in $\mathcal{F}'$.  Hence $\{x,z\}\in\mathcal{F}'$.
  Since we avoid a $U_2$-minor, it must be the case that the last
  possible feasible set, $\{z\}$, is in $\mathcal{F}'$.  Now $D'\ba y$
  is isomorphic to $U_1$, a \cn.  Thus~\ref{notA} holds.

  By~\ref{top}(2) and~\ref{notA}, we know that $A\notin\mathcal{F}$
  and $A-a\in\mathcal{F}$ for some element $a\in A$. The minimality of
  $k$ and having $A\in \mathcal{B}(H^{k-1})-\mathcal{F}(N_{k-1})$
  imply that $N_{k-1}$ is not proper.  Hence no set in $\mathcal{F}$
  has cardinality $|A|$.  If $A\cup\{x,y\}\in\mathcal{F}$, then
  Axiom~(SE) applied to $(A\cup\{x,y\},A-a,y)$ implies that some set
  in $\{A\cup x, A,(A\cup x)-a\}$ is in $\mathcal{F}$.  The
  cardinality of the last two sets is equal to $|A|$, so neither of
  these is in $\mathcal{F}$, and the first also does not occur.  Hence
  $A\cup\{x,y\}\notin\mathcal{F}$.  By~\ref{top}(1),
  $A\cup\{x,y,z\}\in\mathcal{F}$ for some $z\in E-(A\cup\{x,y\})$.

  Let $D'=(E',\mathcal{F}')= (D/(A-a))|\{a,x,y,z\}$.  We know that
  $\mathcal{F}'$ contains $\emptyset,\{a,y\}$, and $\{a,x,y,z\}$, and
  avoids $\{a,x\}$ and $\{a,x,y\}$.  Furthermore $\mathcal{F}'$
  contains no single-element sets since $\mathcal{F}$ contains no sets
  of cardinality $|A|$.  As $D'/\{a,y\}$ is not isomorphic to $U_1$,
  $\{a,y,z\}\notin\mathcal{F}'$.  If $\{a,x,z\}\in\mathcal{F}'$, then
  Axiom~(SE) applied to $(\{a,x,z\},\emptyset ,z)$ implies that a set
  in $\{\{a,x\},\{a\},\{x\}\}$ is in $\mathcal{F}'$, a \cn.  If $D'$
  is even, then it is straightforward to check that it is isomorphic
  to a set system in $\{U_3,U_4,\dots,U_7,T_5,T_6,T_7,T_7^*\}$, a \cn.
  We have ruled out all singleton sets and all three-element sets from
  being in $\mathcal{F}'$ except possibly $\{x,y,z\}$.  Hence
  $\{x,y,z\}\in\mathcal{F}'$.  Now Axiom~(SE) applied to
  $(\{x,y,z\},\emptyset,z)$ implies that some set in
  $\{\{x,y\},\{x\},\{y\}\}$ is in $\mathcal{F}'$.  Hence
  $\{x,y\}\in\mathcal{F}'$ and $D'/\{x,y\}$ is isomorphic to $U_1$, a
  \cn.
\end{proof}

\setcounter{section}{5}
\setcounter{thm}{0}

Next we prove the following excluded-minor characterization of
delta-matroids.

\begin{thm}\label{exdelta}
  A proper set system $S$ is a delta-matroid if and only if $S$ has no
  minor isomorphic to a set system in $\mathcal{S}\cup\mathcal{T}$.
\end{thm}

Recall from Corollary~\ref{cor32} that any delta-matroid may be
obtained from a full Higgs lift delta-matroid by removing some
feasible sets.  Theorem~\ref{exdelta} identifies those intervals that
we must not create when removing feasible sets from Higgs lift
delta-matroids in order to get general delta-matroids.  We note that
$\mathcal{T}$ contains $51$ set systems, which are all shown in
Tables~\ref{t1}--\ref{t8} in the appendix, Section~\ref{sec:app}.  We
will exploit Theorem~\ref{exdelta} and these tables in
Section~\ref{sec:msdm}, where we consider delta-matroids that are
built from matroids.

\begin{proof}[Proof of Theorem~\ref{exdelta}.]
  Every minor of a delta-matroid is a delta-matroid.  Therefore no
  delta-matroid has any minor in $\mathcal{S}\cup\mathcal{T}$.  It is
  easy to check that each single-element deletion and single-element
  contraction of a set system in $\mathcal{S}\cup\mathcal{T}$ is a
  delta-matroid.
  
  Suppose that a proper set system $S=(E,\mathcal{F})$ is an excluded
  minor for the class of delta-matroids.  Then it is not a
  delta-matroid but every minor of $S$, other than $S$ itself, is a
  delta-matroid.  Take sets $A$ and $B$ in $\mathcal{F}$ and element
  $a$ in $A\btu B$ such that $A\btu \{a,x\}$ is not in $\mathcal{F}$
  for all $x\in A\btu B$.  We assume that $|A\btu B|$ is minimized
  fitting this condition.  Up to taking partial duals of $S$, we may
  assume that $B\subset A$.  By deleting the elements in $E-A$ and
  contracting the elements in $B$, we get a minor of $S$ that also
  fails to be a delta-matroid, since Axiom~(SE) fails for the triple
  $(A-B,\emptyset ,a)$.  Thus, we can take $E=A$ and $B=\emptyset$.
  Then $a\in A$, and $A-\{a,x\}\notin \mathcal{F}$ for all $x\in A$.
  Thus $|A|\geq 3$.

  Suppose $A-x$ is in $\mathcal{F}$ for some element $x\in A$.
  Clearly $x\neq a$.  By minimality of $|A\btu B|$, Axiom~(SE) applied
  to the triple $(A-x,\emptyset ,a)$ implies that
  $(A-x)\btu\{a,y\}\in\mathcal{F}$ for some element $y\in A-x$.  As
  $A-\{a,x\}$ is not in $\mathcal{F}$, we know that $y\notin \{x,a\}$,
  and $A-\{a,x,y\}$ is in $\mathcal{F}$ and has three elements fewer
  than $A$.  Furthermore, Axiom~(SE) fails for $(A,A-\{a,x,y\},a)$.
  By the minimality of $A\btu B$, we deduce that $B=A-\{a,x,y\}$, so
  $|A|=3$.  Without loss of generality, $A=\{a,b,c\}$ and $x=c$, so
  $\{a,b\}\in \mathcal{F}$.  Then $\mathcal{F}$ contains
  $A,\{a,b\},\emptyset$, and some sets in $\{\{a\},\{a,c\}\}$.  It
  follows that $S$ is one of $T_1$, $T_2$, $T_3$, or $T_4$.

  We assume then that for all $x\in A$, the set $A-x$ is not in
  $\mathcal{F}$.  Suppose that $A-\{x,y\}$ is in $\mathcal{F}$ for
  some $x,y\in A$.  Clearly $x\neq y$ and $a\notin \{x,y\}$.  Then by
  minimality of $|A\btu B|$, Axiom~(SE) applied to
  $(A-\{x,y\},\emptyset ,a)$ implies that there is an element
  $z\in A-\{x,y\}$ such that $(A-\{x,y\})-\{a,z\}$ is in
  $\mathcal{F}$.  Now Axiom~(SE) does not hold for the triple
  $(A,A-\{a,x,y,z\},a)$, since, for any element $e$ in $\{a,x,y,z\}$,
  the set $A-\{a,e\}$ is not in $\mathcal{F}$.  Thus $|A|\leq 4$.  If
  $|A|<4$, then $|A|=3$, and it is straightforward to check that $S$
  is isomorphic to $T_1^*$. We assume therefore that $|A|=4$, and
  $A=\{a,b,c,d\}$.  Without loss of generality, $\{x,y\}=\{c,d\}$, so
  $\{a,b\}\in\mathcal{F}$.  Now $\mathcal{F}$ does not contain any
  three-element sets, nor does it contain $\{b,c\},\{b,d\}$, or
  $\{c,d\}$.  By the minimality of $|A\btu B|$, Axiom~(SE) holds for
  each triple containing two sets in $\mathcal{F}$ and an element in
  their symmetric difference unless the two sets are $A$ and $B$.  If
  $\{w\}\in\mathcal{F}$ for some $w\in\{b,c,d\}$, then there is an
  element $v\in\{a,b,c,d\}\btu\{w\}$ such that
  $\{a,b,c,d\}\btu\{a,v\}$ is in $\mathcal{F}$.  As no such set is in
  $\mathcal{F}$, we know that $\{a\}$ is the only possible singleton
  set in $\mathcal{F}$.  Therefore, $\mathcal{F}$ contains
  $A, \{a,b\}, \emptyset$, and some sets in
  $\{\{a,c\},\{a,d\},\{a\}\}$.  It is straightforward to check that
  either $S$ is isomorphic to one of $T_5,T_6,T_7$, or $T_8$, or $S/a$
  is isomorphic to $T_1^*$ or $T_2^*$.

  We may now assume that $A-x$ and $A-\{x,y\}$ are not in
  $\mathcal{F}$ for all $x,y\in A$.  Let $A'$ be a second largest set
  in $\mathcal{F}$.  Then Axiom~(SE) fails for the triple $(A,A',e)$,
  for any $e\in A-A'$.  Hence $|A'|=|B|=0$, by minimality of
  $|A\btu B|$.  Let $|A|=k$.  Clearly $k\geq 3$.  Then $S\cong S_k$.
\end{proof}

The next two results are easily obtained from Theorem~\ref{exdelta}.
Both characterize even delta-matroids.  Since minors of even set
systems are even, an even set system is an even delta-matroid if and
only if it does not have, as a minor, any of the even excluded-minors
for delta-matroids.  Let $\mathcal{T}_{5,6,7}$ be the set of all set
systems that are twists of $T_5$, $T_6$, or $T_7$.

\begin{cor}\label{exevendelta}
  A proper, even set system $S$ is an even delta-matroid if and only
  if $S$ has no minor isomorphic to a set system in
  $\{(E,\mathcal{F})\in \mathcal{S}:|E|\text{ is
    even}\}\cup\mathcal{T}_{5,6,7}$.
\end{cor}

The second uses a result of Bouchet~\cite[Lemma~5.4]{ab2}: within the
class of delta-matroids, $S_1$ is the unique excluded minor for even
delta-matroids.  Thus, adjoining $S_1$ to the list of excluded minors
in Theorem~\ref{exdelta} characterizes even-delta matroids.  However,
in order to get the excluded minors, we must discard those other than
$S_1$ that have an $S_1$-minor; those are exactly the set systems in
$\mathcal{T}-\mathcal{T}_{5,6,7}$.

\begin{cor}
\label{exevendelta2}
  A proper set system $S$ is an even delta-matroid if and only
  if $S$ has no minor isomorphic to a set system in
  $\{S_1\}\cup\mathcal{S}\cup\mathcal{T}_{5,6,7}$.
\end{cor}

A delta-matroid is a matroid exactly when its feasible sets are
equicardinal, so it is straightforward to determine the excluded
minors for matroids from Theorem~\ref{exdelta}.

\begin{cor}\label{exmatroids}
  A proper set system $S=(E,\mathcal{F})$ is a matroid if and only if
  all of the sets in $\mathcal{F}$ have the same size, and $S$ has no
  minor isomorphic to a set system in
  \[\{T_5* \{a,d\},\,\, T_6* \{a,d\}\}\cup \{S_{2k}*\{e_1,e_2,\dots
    ,e_k\}:k\geq 2\}.\]
\end{cor}

Excluded-minor characterizations for a number of minor-closed classes
of matroids are known.  For a minor-closed class of matroids
$\mathcal{M}$, let $\Ex(\mathcal{M})$ be its set of excluded minors.
The next corollary follows immediately from
Corollary~\ref{exmatroids}.

\begin{cor}\label{exmsdelta}
  For a minor-closed class of matroids $\mathcal{M}$, a proper set
  system $S=(E,\mathcal{F})$ is in $\mathcal{M}$ if and only if all of
  the sets in $\mathcal{F}$ have the same size and $S$ has no minor
  isomorphic to a set system in
  \[\Ex(\mathcal{M})\cup \{T_5* \{a,d\},\,\, T_6* \{a,d\}\}\cup
  \{S_{2k}* \{e_1,e_2,\dots ,e_k\}:k\geq 2\}.\]
\end{cor}

Let $\mathbb{F}$ be a finite field. For a finite set $E$, let $C$ be a
skew-symmetric $|E|$ by $|E|$ matrix over $\mathbb{F}$, with rows and
columns indexed by the elements of $E$.  Thus, the diagonal of $C$ can
be non-zero only when $\mathbb{F}$ has characteristic two.  Let
$C\left[ A\right]$ be the principal submatrix of $C$ induced by the
set $A \subseteq E$.  Bouchet showed in~\cite{abrep} that we obtain a
delta-matroid, denoted $D(C)$, with ground set $E$ by taking as the
feasible sets all $A\subseteq E$ such that the rank of the matrix
$C[A]$ is $|A|$.  A delta-matroid is called \emph{representable over
  $\mathbb{F}$} if it has a twist that is isomorphic to $D(C)$ for
some skew-symmetric matrix $C$. Note that the empty set is feasible in
$D(C)$.  Thus, for a delta-matroid $D$, if
$D_{\min}\neq (E ,\{\emptyset\})$, then $D$ does not have a matrix
representation.  However, every delta-matroid has a partial dual that
has the empty set as a feasible set; simply take the twist on any
feasible set.  In particular, any matroid $M$ with rank exceeding zero
does not have the empty set as a basis, but, for any basis $B$ of $M$,
the delta-matroid $M*B$ has the empty set among its feasible sets.
The following result by Bouchet~\cite{abrep} shows that, as one would
infer from the common terminology, delta-matroid representability
agrees with matroid representability on the class of matroids.

\begin{prop}\label{binisbin}
  A matroid representable over a field $\mathbb{F}$ is also
  representable over $\mathbb{F}$ as a delta-matroid.
\end{prop}

To be explicit, suppose that a matroid $M$ is representable over
$\mathbb{F}$ and that $B$ is a basis of $M$.  Then $M$ has a
representation of the form $(I|A)$ where $I$ is a $|B| \times |B|$
identity matrix and the columns of $I$ correspond to the elements of
$B$. It is not difficult to see that if
\[ C = \begin{pmatrix} 0 & A \\ -A^T & 0\end{pmatrix},\]
then $M*B = D(C)$.

A delta-matroid representable over the field with two elements is
called \emph{binary}.  Let
\begin{itemize}
\item $P_1=(\{a,b,c\},\{\emptyset , \{a,b\},\{a,c\},
  \{b,c\},\{a,b,c\}\})$;
\item $P_2=(\{a,b,c\},\{\emptyset,\{a\},\{b\}, \{c\}, \{a,b\},\{a,c\},
  \{b,c\}\})$;
\item  $P_3=(\{a,b,c\},\{\emptyset ,\{b\}, \{c\},
  \{a,b\},\{a,c\},\{a,b,c\}\})$;
\item  $P_4=(\{a,b,c,d\},\{\emptyset , \{a,b\},\{a,c\}, \{a,d\},\{b,c\},
  \{b,d\}, \{c,d\}\})$;
\item $P_5=(\{a,b,c,d\},\{\emptyset , \{a,b\},\{a,d\},\{b,c\}, \{c,d\},
\{a,b,c,d\}\})$.
\end{itemize}
Let $\mathcal{P}$ be the set of all twists of $P_1, P_2,P_3,P_4$, or
$P_5$.  In~\cite{rep}, Bouchet and Duchamp proved the following
theorem.

\begin{thm}\label{binarychar}
  A delta-matroid is a binary delta-matroid if and only if it has no minor
  isomorphic to a delta-matroid in $\mathcal{P}$.
\end{thm}

It is worth noting that $P_5*\{a,c\}\cong U_{2,4}$.  Thus the unique excluded
minor for binary matroids is in $\mathcal{P}$, as one would expect.
Combining Theorem~\ref{exdelta} with Bouchet's characterization gives
the following corollary.

\begin{cor}\label{exbindelta}
  A proper set system $S$ is a binary delta-matroid if and only if $S$
  has no minor isomorphic to a set system in
  $\mathcal{P}\cup\mathcal{S} \cup \mathcal{T}$.
\end{cor}

\section{Matroid stack delta-matroids}\label{sec:msdm}

In Section~\ref{sec:hl}, we found that the collection of bases of the
Higgs lifts between a quotient of a matroid $M$ and $M$, or of an
appropriately chosen subcollection of these Higgs lifts, gives a
delta-matroid.  In Section~\ref{sec:lp}, we considered Higgs lifts
between particular pairs of lattice path matroids.  It is natural to
ask, more generally, when a set of matroids can form the layers of a
delta-matroid.  More precisely, suppose we take matroids
$M_1,M_2,\ldots ,M_n$ on $E$ where $1\leq r(M_{i+1})-r(M_i)\leq 2$ for
all $i\in\{1,2,\ldots ,n-1\}$.  Under what circumstances is the set
system
$\bigl(E,\mathcal{B}(M_1)\cup \mathcal{B}(M_2)\cup\cdots \cup
\mathcal{B} (M_n)\bigr)$ a delta-matroid?  This is what we explore in
this section.

Let $S=(E,\mathcal{F})$ be a proper set system where the smallest sets
in $\mathcal{F}$ have cardinality $k$ and the largest have cardinality
$\ell$.  Let $N_i$ be the set system
$(E,\{F: |F|=i \text{ and }F\in\mathcal{F}\})$ for
$i\in\{k, k+1,\dots ,\ell\}$.  We say that $N_k,N_{k+1},\dots ,N_\ell$
is the \emph{stack of $S$}.  If, for some $i$ between $k$ and $\ell$,
no sets in $\mathcal{F}$ have size $i$, then $N_i=(E,\emptyset)$,
which is not proper.  If every proper set system in the stack of $S$
is a matroid, then we say that $S$ is a \emph{matroid stack set
  system}.  Furthermore, if $S$ is a delta-matroid, then we say that
$S$ is a \emph{matroid stack delta-matroid}.  Since the dual of a
matroid is a matroid, it follows that the dual of a matroid stack set
system is also a matroid stack set system, and likewise for a matroid
stack delta-matroid.

We show that if the matroids in the stack of a matroid stack set
system $S$ all belong to a minor-closed class $\mathcal{M}$, then the
proper set systems in the stack of any minor of $S$ all belong to
$\mathcal{M}$. In particular, this implies that the class of matroid
stack delta-matroids is closed under taking minors.

\begin{lemma}\label{matroidstackminors}
  Let $\mathcal{M}$ be a minor-closed class of matroids.  Let
  $S=(E,\mathcal{F})$ be a matroid stack set system where the matroids
  in the stack of $S$ are in $\mathcal{M}$.  If $S'$ is a minor of
  $S$, then $S'$ is a matroid stack set-system, and the matroids in
  the stack of $S'$ are all in $\mathcal{M}$.
\end{lemma}

\begin{proof}
  Take $e\in E$.  It suffices to show that all of the proper set
  systems in the stack of $S\ba e$ and $S/e$ are matroids in
  $\mathcal{M}$.  We consider $S\ba e$ first.  If $e$ is a coloop of
  $S$, then $e$ is a coloop of every matroid in the stack of $S$, and
  the result is clear.  So assume that $e$ is not a coloop.  Let
  $N=(E,\mathcal{F}')$ be a proper set system in the stack of
  $S\ba e$.  The sets in $\mathcal{F}'$ are equicardinal feasible sets
  in $\mathcal{F}(S)$ that avoid $e$, so $S$ has a matroid $M$ in its
  stack such that $\mathcal{F}'\subseteq\mathcal{B}(M)$.  Furthermore,
  the sets in $\mathcal{F}'$ are exactly the bases of $M$ that avoid
  $e$, so $N=M\ba e$.  Thus $N\in\mathcal{M}$.

  Let $\mathcal{M}^*= \{M^*:M \in \mathcal M\}$. Then $\mathcal M^*$
  is a minor-closed class of matroids.  Note that $S^*$ is a matroid
  stack for which each proper set system in the stack belongs to
  $\mathcal M^*$.  Hence the stack of $S^*\ba e$ has all of its proper
  set systems in $\mathcal{M^*}$, and so $(S^*\ba e)^*$ is a matroid
  stack for which each proper set system in the stack belongs to
  $\mathcal M$. This last set system is equal to $S/e$.
\end{proof}

In the next corollary, we use Theorem~\ref{exdelta} to find the
excluded minors for matroid stack delta-matroids within the class of
matroid stack set systems.  The excluded minors are exactly those set
systems in $\mathcal{S}\cup\mathcal{T}$ that are matroid stack set
systems.  Note that any proper set system $(E,\mathcal{F})$ where
$|E|=3$ and the sets in $\mathcal{F}$ are equicardinal is a matroid.
For this reason, every twist of $T_1$, $T_2$, $T_3$, or $T_4$ is an
excluded minor for matroid stack delta-matroids.  Let
$\mathcal{T}_{1,2,3,4}$ be the set of these twists.

\begin{cor}\label{exmatroidstack}
  Let $D$ be a matroid stack set system.  Then $D$ is a matroid stack
  delta-matroid if and only if it contains no minor isomorphic to a
  set system in any of the following sets:
  \begin{enumerate}
  \item $\{S_{k}*X:k\geq 3, X\subseteq E(S_{k}),\text{ and }|X|\neq
    k/2\}$,
  \item $\mathcal{T}_{1,2,3,4}$,
  \item $\{T_5,\,\, T_5* a,\,\, T_5*\{b,c,d\} \}$,
  \item $\{T_6,\,\, T_6*a,\,\, T_6*b,\,\, T_6*\{b,c,d\} \}$,
  \item $\{T_7,\,\, T_7*a,\,\, T_7*b,\,\, T_7*\{a,c,d\},\,\,
    T_7*\{b,c,d\},\,\, T_7^*\}$,
  \item $\{T_8,\,\, T_8*a,\,\, T_8*b,\,\, T_8*\{a,c,d\},\,\,
    T_8*\{b,c,d\},\,\, T_8^*\}$.
  \end{enumerate}
\end{cor}

\begin{proof}
  If $D$ is a matroid stack set system that is not a delta-matroid
  then it must have a minor $D'$ isomorphic to a set system in
  $\mathcal S \cup \mathcal T$. Moreover,
  Lemma~\ref{matroidstackminors} implies that $D'$ must be a matroid
  stack set system. The result follows by checking which elements of
  $\mathcal S \cup \mathcal T$ are matroid stack systems.
\end{proof}

\begin{figure}
\includegraphics{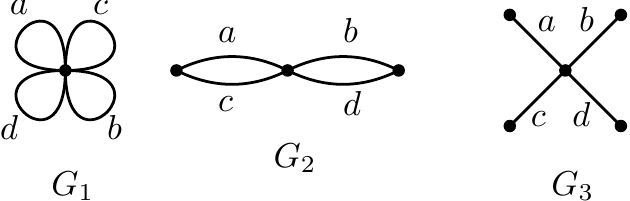}
\caption{The spanning trees of these graphs are the feasible sets of
  $P_5$.}\label{p5graphs}
\end{figure}

Note that representability within the stack of a matroid stack
delta-matroid does not guarantee representability of the delta-matroid.
For example, $P_5$ is an excluded minor for binary delta-matroids,
but it is also a matroid stack delta-matroid where each matroid in the
stack is binary.  In fact, each matroid is graphic, and these graphs are
depicted in Figure~\ref{p5graphs}.

The class of even delta-matroids is minor-closed.  Hence the next
result is a corollary of Lemma~\ref{matroidstackminors}.

\begin{cor}\label{evenms}
  The class of matroid stack delta-matroids that are even is minor-closed
  and dual-closed.
\end{cor}

The following result is easily obtained from
Corollary~\ref{exmatroidstack} by identifying those set systems in the
excluded minors for matroid stack delta-matroids that are even.

\begin{cor}\label{exmatroidstackeven}
  An even matroid stack set system is an even matroid stack
  delta-matroid if and only if it contains no minor isomorphic to a
  set system in one of the following sets:
  \begin{enumerate}
  \item $\{S_{2k}*X:k\geq 2, X\subseteq E(S_{2k}),\text{ and }|X|\neq
    k\}$,
  \item $\{T_5,\,\, T_5*a,\,\, T_5*\{b,c,d\} \}$,
  \item $\{T_6,\,\, T_6*a,\,\, T_6*b,\,\, T_6*\{b,c,d\} \}$,
  \item $\{T_7,\,\, T_7*a,\,\, T_7*b,\,\, T_7*\{b,c,d\},\,\,
    T_7*\{a,c,d\},\,\, T_7^*\}$.
    \end{enumerate}
\end{cor}

We next consider matroid stack delta-matroids where each matroid in
the stack is paving.  A rank-$r$ matroid is \emph{paving} if each of
its circuits has size at least $r$.  Although the class of paving
matroids is closed under minors, it is not closed under duality.  Let
$D$ be a set system where every proper set system in its stack is a
paving matroid.  Then we say that $D$ is a \emph{paving set system}.
If $D$ is also a delta-matroid, then we say that $D$ is a \emph{paving
  delta-matroid}.  The next result follows from
Lemma~\ref{matroidstackminors}.

\begin{cor}\label{pavingminor}
  Every minor of a paving delta-matroid is a paving delta-matroid.
\end{cor}

By identifying the paving set systems among the excluded minors
for matroid stack delta-matroids, we find the excluded minors for
paving delta-matroids.

\begin{cor}\label{expaving}
  A paving set system is a paving delta-matroid if and
  only if it contains no minor isomorphic to a set system in
  the following sets:
  \begin{enumerate}
  \item $\{S_i:i\geq 3\}$,
  \item $\{T_1*\{b,c\},\,\, T_1^*\}$,
  \item $\{T_2,\,\, T_2*\{a,b\},\,\, T_2*\{b,c\},\,\, T_2^*\}$,
  \item $\{T_3*b,\,\, T_3*\{b,c\} \}$,
  \item $\{T_4,\,\, T_4*b,\,\, T_4*\{a,c\},\,\, T_4*\{b,c\}\}$,
  \item $\{T_6*\{b,c,d\}\}$,
  \item $\{T_7,\,\, T_7*b,\,\, T_7*\{b,c,d\}\}$,
  \item $\{T_8,\,\, T_8*\{b,c,d\}\}$.
\end{enumerate}
\end{cor}

Next we consider matroid stack delta-matroids where each matroid in
the stack is a sparse paving matroid.  A matroid is \emph{sparse
  paving} if it is paving and its dual is paving.  Equivalently, a
matroid is sparse paving if each non-spanning circuit is a hyperplane.
Let $D$ be a set system where every proper set system in its stack is
a sparse paving matroid.  Then we say that $D$ is a \emph{sparse
  paving set system}.  If $D$ is also a delta-matroid, then we say
that $D$ is a \emph{sparse paving delta-matroid}.  It is easy to see
that every minor of a sparse paving matroid is sparse paving.  Note
that the class of sparse paving delta-matroids is closed under
duality.  Hence the next result follows immediately from
Lemma~\ref{matroidstackminors}.

\begin{cor}\label{sparseminor}
  The class of sparse paving delta-matroids is minor-closed
  and dual-closed.
\end{cor}

In~\cite{asymp}, it is conjectured that, asymptotically, almost all
matroids are sparse paving.  That is, if $sp(n)$ is the number of
sparse paving matroids with $n$ elements, and $m(n)$ is the number of
matroids with $n$ elements, then it is conjectured that
$\lim_{n\to\infty}\frac{sp(n)}{m(n)}=1$.  We make the following
related conjecture.

\begin{conj}
  Asymptotically, almost all matroid stack delta-matroids are sparse
  paving.
\end{conj}

In a similar vein, one might wonder if, asymptotically, almost all
delta-matroids are sparse paving, but this is far from being true, as
the number of delta-matroids is significantly greater than the number
of matroid stack delta-matroids.  It is shown in~\cite{FMN} that the
number $d_n$ of delta-matroids with ground set $\{1,\ldots,n\}$ is at
least $2^{2^{n-1}}$. On the other hand, in~\cite{BPV} it is shown that
the number $m_n$ of matroids with ground set $\{1,\ldots,n\}$
satisfies $\log \log m_n \leq n-\frac{3}{2} \log n + O(1)$, where all
logs are taken to base $2$.  A crude estimate gives an upper bound of
$f_n =(m_n+1)^{n+1}$ for the number of matroid stack delta-matroids
with ground set $\{1,\ldots,n\}$ and
$\log \log f_n = n - \frac{1}{2} \log n + O(1) < n-1\leq\log \log
d_n$.

Repeating this analysis for even delta-matroids yields a different
picture, as it is also shown in~\cite{FMN} that the number $e_n$ of
even delta-matroids with ground set $\{1,\ldots,n\}$ satisfies
\[ n-1-\log n \leq \log \log e_n \leq n-\log n + O(\log \log n),\]
with the lower bound being the number of even sparse paving
delta-matroids with ground set $\{1,\ldots,n\}$, so we pose the
following open question.

\begin{open}
  Asymptotically, are almost all even delta-matroids sparse paving?
\end{open}

It is straightforward to identify the sparse paving set systems that
are excluded minors for matroid stack delta-matroids.  These comprise
the excluded minors for sparse paving delta-matroids within the class
of sparse paving set systems.  Since the class of sparse paving
delta-matroids is closed under duality, every set system in the set of
excluded minors for sparse paving delta-matroids has its dual also in
the list of excluded minors.

\begin{cor}\label{exsparse}
  A sparse paving set system is a sparse paving delta-matroid if and
  only if it contains no minor isomorphic to a set system in
  \[\{S_i:i\geq 3\}\cup \{T_2,\,\, T_2^*,\,\, T_3*b,\,\, T_4*b,\,\,
    T_4*\{a,c\} \}.\]
\end{cor}

Next we consider matroid stack delta-matroids with a stack of
quotients.  That is, let $D$ be a matroid stack set system.  If every
matroid in the stack is a quotient of the matroid with next highest
rank in the stack, then we say that $D$ is a \emph{quotient set
  system}.  It follows from Lemma~\ref{lem:quotviabases} that a
quotient of a quotient of $M$ is also a quotient of $M$.  Therefore
every matroid in the stack of a quotient set system is a quotient of
every matroid in the stack with higher rank.  If $D$ is a quotient set
system, and $D$ is a delta-matroid, then we say that it is a
\emph{quotient delta-matroid}. It also follows from
Lemma~\ref{lem:quotviabases} that $D^*$ is also a quotient
delta-matroid.

\begin{lemma}\label{quotientminors}
  The class of quotient delta-matroids is minor-closed and dual-closed.
\end{lemma}

\begin{proof}
  Let $D=(E,\mathcal{F})$ be a quotient delta-matroid and take
  $e \in E$.  Since $D/e=(D^*\ba e)^*$, it suffices to show that
  $D\ba e$ is a quotient delta-matroid.  By
  Lemma~\ref{matroidstackminors}, $D\ba e$ is a matroid stack
  delta-matroid.  The matroids $M$ and $M'$ in the stack of $D\ba e$
  are obtained from some matroids $N$ and $N'$ in the stack of $D$ by
  deleting $e$.  Without loss of generality, we assume that $N$ is a
  quotient of $N'$.  Then, by Lemma~\ref{lem:higgsdual}, $M$ is a
  quotient of $M'$.
\end{proof}

Note that if $M$ is the matroid with ground set $\{1,2,3,4\}$ and set
of bases $$\{\{1,2\},\{1,3\}, \{1,4\},\{2,3\},\{2,4\}\},$$ then
$M*\{1,3\}$ is not a matroid stack delta-matroid.  Therefore none of
the classes of matroid stack delta-matroids, sparse paving
delta-matroids, or quotient delta-matroids is closed under twists.

The following result is easily obtained from
Corollary~\ref{exmatroidstack} by identifying the quotient set systems
in $\mathcal{S}\cup\mathcal{T}$.

\begin{cor}\label{exquotientstack}
  A quotient set system is a quotient delta-matroid if and only if it
  does not have a minor in
  \[\{S_i:i\geq 3\}\cup\{T_1,\,\, T_1^*,\,\, T_2,\,\,
  T_2^*,\,\, T_3,\,\, T_4,\,\, T_4^*,\,\, T_5,\,\, T_6,\,\, T_7,\,\,
  T_7^*,\,\, T_8,\,\, T_8^*\}.\]
\end{cor}

We note the following two properties of even sparse paving set
systems. A simple generalization of Lemma~4.1 from~\cite{FMN} shows
that if the stack of an even sparse paving set system $S$ contains no
improper set systems other than those required to ensure evenness,
then $S$ is a delta-matroid.  Moreover an even sparse paving set
system is also a quotient set system.  Hence we have the following
proposition.
\begin{prop}\label{speven}
  If $S$ is an even sparse paving set system, then $S$ is a quotient
  set system.
\end{prop}

\section{Appendix: The twists of $T_1,T_2,\ldots,T_8$}\label{sec:app}

\begin{table}[H]
\begin{center}
\begin{tabular}{r|cccc|cccc|l}
  \cline{2-9}
  $T_1$ & $\emptyset$ &  & $\{a,b\}$ & $\{a,b,c\}$ & $\emptyset$ & $\{c\}$ &
  & $\{a,b,c\}$\rule{0pt}{10pt} & $T_1^*$ \\
  \cline{2-9}
 \multicolumn{10}{c}{}\\[-0.7em]
  \cline{2-9}
  \multirow{2}{*}{$T_1*\{a\}$} & & $\{a\}$ & \multirow{2}{*}{$\{b,c\}$} &  &  &
  \multirow{2}{*}{$\{a\}$} & $\{b,c\}$ &
  &  \multirow{2}{*}{$T_1*\{b,c\}$}\rule{0pt}{10pt}  \\
\hspace{1.6cm} & & $\{b\}$ & & & & &  $\{a,c\}$ & &  \hspace{1.6cm} \\
  \cline{2-9}
 \multicolumn{10}{c}{}\\[-0.7em]
  \cline{2-9}
  $T_1*\{c\}$ &   & $\{c\}$ & $\{a,b\}$ & $\{a,b,c\}$ & $\emptyset$ & $\{c\}$ &
     $\{a,b\}$ &  \rule{0pt}{10pt} &  $T_1*\{a,b\}$ \\
  \cline{2-9}  \multicolumn{10}{c}{}\\[-0.7em]
\end{tabular}
\end{center}
\caption{All twists of $T_1$ up to isomorphism.  Dual pairs
  are side by side.}\label{t1}
\end{table}

\vfill

\begin{table}[H]
\begin{center}
\begin{tabular}{r|cccc|cccc|l}
 \cline{2-9}
  \multirow{2}{*}{$T_2$}  & \multirow{2}{*}{$\emptyset$}
  &  & $\{a,b\}$ &  \multirow{2}{*}{$\{a,b,c\}$} &
 \multirow{2}{*}{$\emptyset$} & $\{c\}$ &  &
 \multirow{2}{*}{$\{a,b,c\}$}\rule{0pt}{10pt} &
 \multirow{2}{*}{$T_2^*$}  \\
& &  & $\{a,c\}$ &  &  & $\{b\}$ &  & & \\
 \cline{2-9}
 \multicolumn{10}{c}{}\\[-0.7em]
 \cline{2-9}
 & & $\{a\}$ &  &  &  &  & $\{b,c\}$ &  \rule{0pt}{10pt} & \\
 $T_2*\{a\}$ & & $\{b\}$ & $\{b,c\}$ &  &  & $\{a\}$ & $\{a,c\}$ & &
  $T_2*\{b,c\}$ \\
\hspace{1.6cm} & & $\{c\}$ &  &  &  &  & $\{a,b\}$ & &  \hspace{1.6cm} \\
 \cline{2-9}
 \multicolumn{10}{c}{}\\[-0.7em]
 \cline{2-9}
  \multirow{2}{*}{$T_2*\{c\}$} & & $\{a\}$ &  \multirow{2}{*}{$\{a,b\}$} &
  \multirow{2}{*}{$\{a,b,c\}$} & \multirow{2}{*}{$\emptyset$} &
  \multirow{2}{*}{$\{c\}$} & $\{b,c\}$ &  \rule{0pt}{10pt} &
  \multirow{2}{*}{$T_2*\{a,b\}$}\\
  & & $\{c\}$ &  &  &  &  & $\{a,b\}$ & & \\
 \cline{2-9} \multicolumn{10}{c}{}\\[-0.7em]
\end{tabular}
\end{center}
\caption{All twists of $T_2$ up to isomorphism.}\label{t2}
\end{table}

\vfill

\begin{table}[H]
\begin{center}
\begin{tabular}{r|cccc|l}
  \cline{2-5}
 $T_3$ &  $\emptyset$ & $\{a\}$ & $\{a,b\}$ &
 $\{a,b,c\}$\rule{0pt}{10pt} & \\
 \cline{2-5}
 \multicolumn{6}{c}{}\\[-0.7em]
 \cline{2-5}
  \multirow{2}{*}{$T_3*\{b\}$} & & $\{a\}$ & $\{a,b\}$ & & \rule{0pt}{10pt} \\
 \hspace{1.6cm} & & $\{b\}$ & $\{a,c\}$ & & \hspace{1.6cm} \\
   \cline{2-5}
 \multicolumn{6}{c}{}\\[-0.8em]
\end{tabular}     \\
\begin{tabular}{r|cccc|cccc|l}
 \cline{2-9}
  \multirow{2}{*}{$T_3*\{a\}$} & \multirow{2}{*}{$\emptyset$} &
  $\{a\}$ &  \multirow{2}{*}{$\{b,c\}$}  &
  \hspace{35pt}  &  \hspace{5pt} &  \multirow{2}{*}{$\{a\}$} &
  $\{b,c\}$\rule{0pt}{10pt} &  \multirow{2}{*}{$\{a,b,c\}$} &
  \multirow{2}{*}{$T_3*\{b,c\}$}\\  \hspace{1.6cm}
           & & $\{b\}$ &  &  &  &  & $\{a,c\}$ & & \hspace{1.6cm} \\
   \cline{2-9} \multicolumn{10}{c}{}\\[-0.7em]
\end{tabular}
\end{center}
\caption{All twists of $T_3$ up to isomorphism.  A twist alone in a
  row is self-dual.}\label{t3}
\end{table}

\vfill

\begin{table}[H]
\begin{center}
\begin{tabular}{r|cccc|cccc|l}
   \cline{2-9}
 \multirow{2}{*}{$T_4$} & \multirow{2}{*}{$\emptyset$} &
 \multirow{2}{*}{$\{a\}$} & $\{a,b\}$  & \multirow{2}{*}{$\{a,b,c\}$}
 & \multirow{2}{*}{$\emptyset$} & $\{c\}$ & \multirow{2}{*}{$\{b,c\}$}
 &  \multirow{2}{*}{$\{a,b,c\}$} &   \multirow{2}{*}{$T_4^*$} \\
 \hspace{1.6cm} & &  & $\{a,c\}$ &  &  & $\{b\}$ &  & & \hspace{1.6cm}
  \\ \cline{2-9} \multicolumn{10}{c}{}\\[-0.7em] \cline{2-9}
 & & $\{a\}$ &  &  &  &  & $\{b,c\}$ & & \\
 $T_4*\{a\}$ & $\emptyset$ & $\{b\}$ & $\{b,c\}$ &  &  & $\{a\}$ &
 $\{a,c\}$ & $\{a,b,c\}$ &  $T_4*\{b,c\}$ \\
 & & $\{c\}$ &  &  &  &  & $\{a,b\}$ & & \\
 \cline{2-9} \multicolumn{10}{c}{}\\[-0.7em] \cline{2-9}
  \multirow{2}{*}{$T_4*\{b\}$}   & & $\{a\}$ & $\{a,b\}$ &
  \multirow{2}{*}{$\{a,b,c\}$} & \multirow{2}{*}{$\emptyset$}
  & $\{c\}$ & $\{b,c\}$ & &
  \multirow{2}{*}{$T_4*\{a,c\}$}     \\
 & & $\{b\}$ & $\{a,c\}$ &  &  & $\{b\}$ & $\{a,c\}$ & &  \\  \cline{2-9}
 \multicolumn{10}{c}{}\\[-0.7em]
\end{tabular}
\end{center}
\caption{All twists of $T_4$ up to isomorphism.}\label{t4}
\end{table}

\vfill

\begin{table}[H]
\begin{center}
\begin{tabular}{r|ccccc|l}
  \cline{2-6}
  $T_5$  & $\emptyset$ & \hspace{15pt}  &
  $\{a,b\}$ &  \hspace{35 pt} & $\{a,b,c,d\}$\rule{0pt}{10pt} &
  \\ \cline{2-6} \multicolumn{6}{c}{}\\[-0.7em] \cline{2-6}
 & & & $\{a,d\}$ & & & \rule{0pt}{10pt} \\
  $T_5*\{a,d\}$ & & & $\{b,d\}$ & & & \\
 \hspace{1.6cm} & & & $\{b,c\}$ & & &  \hspace{1.6cm} \\  \cline{2-6}
 \multicolumn{6}{c}{}\\[-0.7em]
\end{tabular}
\begin{tabular}{|ccccc|ccccc|}
  \hline & $\{a\}$ & & \multirow{2}{*}{$\{b,c,d\}$} & &  &
      \multirow{2}{*}{$\{a\}$} & & $\{b,c,d\}$
       &  \rule{0pt}{10pt}  \\
         & $\{b\}$ & &
        & \hspace{45pt} & & & & $\{a,c,d\}$ & \\
  \hline \multicolumn{5}{c}{$T_5*\{a\}$} &
   \multicolumn{5}{c}{$T_5*\{b,c,d\}$}\rule[-7pt]{0pt}{17pt} \\ \hline
  \multirow{2}{*}{$\emptyset$}  & & $\{a,b\}$ & & & & & $\{c,d\}$ & &
  \multirow{2}{*}{$\{a,b,c,d\}$}\rule{0pt}{10pt} \\
  & & $\{c,d\}$ & & & & & $\{a,b\}$ & &  \\  \hline
  \multicolumn{5}{c}{$T_5*\{a,b\}$} &
  \multicolumn{5}{c}{$T_5*\{c,d\}$}\rule[-7pt]{0pt}{18pt}  \\
 \multicolumn{10}{c}{}\\[-0.7em]
\end{tabular}
\end{center}
\caption{All twists of $T_5$ up to isomorphism.}\label{t5}
\end{table}

\vfill

\begin{table}[H]
\begin{center}
\begin{tabular}{r|ccccc|l}
  \cline{2-6}
\multirow{2}{*}{$T_6$}  & \multirow{2}{*}{$\emptyset$} & \hspace{15pt} & $\{a,b\}$ &
   \hspace{35 pt} &  \multirow{2}{*}{$\{a,b,c,d\}$}\rule{0pt}{10pt} &  \\
 & & & $\{a,c\}$ & & &\\   \cline{2-6}
 \multicolumn{7}{c}{}\\[-0.7em]
  \cline{2-6}
  \multirow{2}{*}{$T_6*\{b\}$} &  & $\{a\}$ & & $\{a,b,c\}$ & & \rule{0pt}{10pt} \\
&  & $\{b\}$ & & $\{a,c,d\}$ & &   \\   \cline{2-6}
 \multicolumn{7}{c}{}\\[-0.7em]
    \cline{2-6}
& & & $\{a,d\}$ & & & \rule{0pt}{10pt} \\
 \multirow{2}{*}{$T_6*\{a,d\}$} & & & $\{b,c\}$ & & & \\
& & & $\{b,d\}$ & & &\\
 \hspace{1.6cm} & & & $\{c,d\}$ & & &  \hspace{1.6cm} \\   \cline{2-6}
 \multicolumn{7}{c}{}\\[-0.7em]
\end{tabular}
\begin{tabular}{|ccccc|ccccc|}
  \hline & $\{a\}$ & & & & & & & $\{b,c,d\}$ & \rule{0pt}{10pt}  \\
  & $\{b\}$ & & $\{b,c,d\}$ & \hspace{45pt} & & $\{a\}$ & & $\{a,c,d\}$ & \\
  & $\{c\}$ & & & & & & & $\{a,b,d\}$ & \\
  \hline \multicolumn{5}{c}{$T_6*\{a\}$} &
  \multicolumn{5}{c}{$T_6*\{b,c,d\}$}\rule[-7pt]{0pt}{17pt}  \\ \hline
  & & $\{a,b\}$ & & & & & $\{c,d\}$ & & \rule{0pt}{10pt} \\
  $\emptyset$ & & $\{b,c\}$ & & & & & $\{a,d\}$ &  &  $\{a,b,c,d\}$  \\
  & & $\{c,d\}$ & & & & & $\{a,b\}$ & &\\
  \hline
  \multicolumn{5}{c}{$T_6*\{a,b\}$} &
  \multicolumn{5}{c}{$T_6*\{c,d\}$}\rule[-7pt]{0pt}{18pt} \\
 \multicolumn{10}{c}{}\\[-0.7em]
\end{tabular}
\end{center}
\caption{All twists of $T_6$ up to isomorphism.}\label{t6}
\end{table}

\vfill

\begin{table}[H]
\begin{center}
\begin{tabular}{|ccccc|ccccc|}
  \hline
  & & $\{a,b\}$ & & & & & $\{c,d\}$ & &  \rule{0pt}{10pt} \\
  $\emptyset$ & & $\{a,c\}$ & & $\{a,b,c,d\}$ &
  $\emptyset$ & & $\{b,d\}$ & & $\{a,b,c,d\}$\\
  & & $\{a,d\}$ & & & & & $\{b,c\}$ & &\\ \hline
  \multicolumn{5}{c}{$T_7$} & \multicolumn{5}{c}{$T_7^*$}\rule[-7pt]{0pt}{18pt}  \\
  \hline  & $\{a\}$ & & & & & \multirow{4}{*}{$\{a\}$} & & $\{b,c,d\}$
       &  \rule{0pt}{10pt}  \\
  & $\{b\}$ & & \multirow{2}{*}{$\{b,c,d\}$} & & & & & $\{a,c,d\}$ & \\
  & $\{c\}$ & & & & & & & $\{a,b,d\}$ & \\
  & $\{d\}$ & & & & & & & $\{a,b,c\}$ & \\
  \hline \multicolumn{5}{c}{$T_7*\{a\}$} &
  \multicolumn{5}{c}{$T_7*\{b,c,d\}$}\rule[-7pt]{0pt}{17pt}  \\ \hline
  &   \multirow{2}{*}{$\{a\}$} & & $\{a,b,c\}$ & & & $\{d\}$ & &
   \multirow{2}{*}{$\{b,c,d\}$}  &   \rule{0pt}{10pt} \\
  &  \multirow{2}{*}{$\{b\}$} & & $\{a,b,d\}$ & &
  & $\{c\}$ & &  \multirow{2}{*}{$\{a,c,d\}$} & \\
  & & & $\{a,c,d\}$ & & & $\{b\}$ & & &\\
  \hline
  \multicolumn{5}{c}{$T_7*\{b\}$} &
  \multicolumn{5}{c}{$T_7*\{a,c,d\}$}\rule[-7pt]{0pt}{18pt} \\  \hline
  & & $\{a,b\}$ & & & & &  $\{c,d\}$ & &  \rule{0pt}{10pt}  \\
  \multirow{2}{*}{$\emptyset$} & &  $\{b,c\}$ & & &
                      & &  $\{a,d\}$ & &  \multirow{2}{*}{$\{a,b,c,d\}$}\\
  & & $\{b,d\}$ & & & & & $\{a,c\}$ & & \\
  & & $\{c,d\}$ & & & & &  $\{a,b\}$ & & \\
  \hline
  \multicolumn{5}{c}{$T_7*\{a,b\}$} &
  \multicolumn{5}{c}{$T_7*\{c,d\}$}\rule[-7pt]{0pt}{17pt}
  \\
 \multicolumn{10}{c}{}\\[-0.7em]
\end{tabular}
\end{center}
\caption{All twists of $T_7$ up to isomorphism.}\label{t7}
\end{table}

\vfill

\begin{table}[H]
\begin{center}
\begin{tabular}{|ccccc|ccccc|}
  \hline
  & & $\{a,b\}$ & & & & & $\{c,d\}$ & &  \rule{0pt}{10pt} \\
  $\emptyset$ & $\{a\}$ & $\{a,c\}$ & & $\{a,b,c,d\}$ &
  $\emptyset$ & & $\{b,d\}$ & $\{b,c,d\}$ & $\{a,b,c,d\}$\\
  & & $\{a,d\}$ & & & & & $\{b,c\}$ & &\\ \hline
  \multicolumn{5}{c}{$T_8$} & \multicolumn{5}{c}{$T_8^*$}\rule[-7pt]{0pt}{18pt}  \\
  \hline  & $\{a\}$ & & & & & \multirow{4}{*}{$\{a\}$} & & $\{b,c,d\}$
       &  \rule{0pt}{10pt}  \\
  \multirow{2}{*}{$\emptyset$}  & $\{b\}$ & & \multirow{2}{*}{$\{b,c,d\}$}
   & & & & & $\{a,c,d\}$ & \multirow{2}{*}{$\{a,b,c,d\}$} \\
  & $\{c\}$ & & & & & & & $\{a,b,d\}$ & \\
  & $\{d\}$ & & & & & & & $\{a,b,c\}$ & \\
  \hline \multicolumn{5}{c}{$T_8*\{a\}$} &
  \multicolumn{5}{c}{$T_8*\{b,c,d\}$}\rule[-7pt]{0pt}{17pt}  \\ \hline
  &   \multirow{2}{*}{$\{a\}$} & & $\{a,b,c\}$ & & & $\{d\}$ & &
   \multirow{2}{*}{$\{b,c,d\}$}  &   \rule{0pt}{10pt} \\
  &  \multirow{2}{*}{$\{b\}$} & $\{a,b\}$ & $\{a,b,d\}$ & &
  & $\{c\}$ & $\{c,d\}$ &  \multirow{2}{*}{$\{a,c,d\}$} & \\
  & & & $\{a,c,d\}$ & & & $\{b\}$ & & &\\
  \hline
  \multicolumn{5}{c}{$T_8*\{b\}$} &
  \multicolumn{5}{c}{$T_8*\{a,c,d\}$}\rule[-7pt]{0pt}{18pt}  \\
  \hline
  & & $\{a,b\}$ & & & & &  $\{c,d\}$ & &  \rule{0pt}{10pt}  \\
  \multirow{2}{*}{$\emptyset$} &  \multirow{2}{*}{$\{b\}$} &  $\{b,c\}$ & & &
  & &  $\{a,d\}$ &  \multirow{2}{*}{$\{a,c,d\}$} &  \multirow{2}{*}{$\{a,b,c,d\}$}\\
  & & $\{b,d\}$ & & & & & $\{a,c\}$ & & \\
  & & $\{c,d\}$ & & & & &  $\{a,b\}$ & & \\
  \hline
  \multicolumn{5}{c}{$T_8*\{a,b\}$} &
  \multicolumn{5}{c}{$T_8*\{c,d\}$}\rule[-7pt]{0pt}{17pt}  \\
 \multicolumn{10}{c}{}\\[-0.7em]
\end{tabular}
\end{center}
\caption{All twists of $T_8$ up to isomorphism.}\label{t8}
\end{table}

\vspace{5pt}

\begin{center}
 \textsc{Acknowledgments}
\end{center}

\vspace{3pt}

The authors thank the referee for comments that improved the
exposition.  J.\ Bonin thanks Vic Reiner for discussions that led to
discovering special cases of what later grew into the lattice path
delta-matroids that are introduced here.


\begin{thebibliography}{99}

\bibitem{lpmfacial} S.~An, J.~Jung, and S.~Kim, Facial structures of
  lattice path matroid polytopes, preprint, arXiv:1701.00362.

\bibitem{BPV} N.~Bansal. R.~A. Pendavingh, and J.~G. van der Pol, On
  the number of matroids, \emph{Combinatorica} \textbf{35} (2015)
  253--277.

\bibitem{MR2718679} J.~Bonin, Lattice path matroids: the excluded
  minors, \emph{J.\ Combin.\ Theory Ser.\ B} \textbf{100} (2010)
  585--599.

\bibitem{omer} J.~Bonin and O.~Gim\'enez, Multi-path matroids,
  \emph{Combin.\ Probab.\ Comput.}\ \textbf{16} (2007) 193--217.

\bibitem{lpm2} J.~Bonin and A.~de Mier, Lattice path matroids:
  structural properties, \emph{European J.\ Combin.}\ \textbf{27}
  (2006) 701--738.

\bibitem{lpm1} J.~Bonin, A.~de Mier, and M.~Noy, Lattice path
  matroids: enumerative aspects and Tutte polynomials, \emph{J.\
    Combin.\ Theory Ser.\ A} \textbf{104} (2003) 63--94.

\bibitem{splice} J.~Bonin and W.~Schmitt, Splicing matroids,
  \emph{European J.\ Combin.}\ \textbf{32} (2011) 722--744.

\bibitem{ab1} A.~Bouchet, Greedy algorithm and symmetric matroids,
  \emph{Math.\ Program.}\ \textbf{38} (1987) 147--159.

\bibitem{ab2} A.~Bouchet, Maps and delta-matroids, \emph{Discrete
    Math.}\  \textbf{78} (1989) 59--71.

\bibitem{abrep} A.~Bouchet, Representability of $\Delta$-matroids, in
  \emph{Proceedings of the 6th Hungarian Colloquium of
    Combinatorics}, Colloq.\ Math.\ Soc.\ J\'anos Bolyai, (1987)
  167--182.

\bibitem{multi2} A.~Bouchet, Multimatroids II. Orthogonality,
  minors and connectivity, \emph{Electron.\ J. Combin.}\ \textbf{5}
  (1998), Paper 8, 25 pp.

\bibitem{rep} A.~Bouchet and A.~Duchamp, Representability of
  delta-matroids over GF(2), \emph{Linear Algebra Appl.}\ {\bf 146}
  (1991) 67--78.

\bibitem{Brylawski} T.H. Brylawski, Constructions, in: \emph{Theory of
    Matroids}, N.~White ed.\ Cambridge Univ.\ Press, Cambridge, (1986)
  127--223.

\bibitem{CMNR} C.~Chun, I.~Moffatt, S.~D. Noble, and R.~Rueckriemen,
  Matroids, delta-matroids and embedded graphs, preprint,
  arXiv:1403.0920v2.

  
\bibitem{CMNR2} C.~Chun, I.~Moffatt, S.~D. Noble, and R.~Rueckriemen,
  On the interplay between embedded graphs and delta-matroids,
  \emph{Proc.\ Lond.\ Math.\ Soc.}\ in press.

\bibitem{lpmfm} E.~Cohen, P.~Tetali, and D.~Yeliussizov, Lattice path
  matroids: negative correlation and fast mixing, preprint,
  arXiv:1505.06710.

\bibitem{MR3413585} E.~Delucchi and M.~Dlugosch, Bergman complexes
  of lattice path matroids, \emph{SIAM J. Discrete Math.}\
  \textbf{29} (2015) 1916--1930.

\bibitem{saturateddelta} C.~Dupont, A.~Fink, and L.~Moci, Universal
  Tutte characters via combinatorial coalgebras, \emph{Algebr.\
    Comb.}\ to appear.

\bibitem{lpmtoric} S.~Eu, Y.~Lo, and Y.~Tsai, Toric $g$-polynomials of
  hook shape lattice path matroid polytopes and product of simplices,
  preprint, arXiv:1508.04674.

\bibitem{FMN} D.~Funk, D.~Mayhew and S.~Noble, How many delta-matroids
  are there? preprint, \emph{European J.\ Combin.}\ \textbf{69} (2018)
  149--158.

\bibitem{lpmpolytope} K.~Knauer, L.~Mart\'inez-Sandoval, and
J.~L.~Ramírez Alfons\'in, On lattice path matroid polytopes: integer
  points and Ehrhart polynomial, \emph{Discrete Comput.\ Geom.}\ (2018).
%
%

\bibitem{lpmTPineq} K.~Knauer, L.~Mart\'inez-Sandoval, and
  J.~L.~Ramírez Alfons\'in, A Tutte polynomial inequality for lattice
  path matroids, \emph{Adv.\ in Appl.\ Math.}\ \textbf{94} (2018)
  23--38.

\bibitem{Strong} J.~P. S. Kung, Strong maps, in: \emph{Theory of
    Matroids}, N.~White ed.\ Cambridge Univ.\ Press, Cambridge, (1986)
  224--253.

\bibitem{asymp} D.~Mayhew, M.~Newman, D.~Welsh, and G.~Whittle, On the
  asymptotic proportion of connected matroids, \emph{European J.\
    Combin.}\ \textbf{32} (2011) 882--890.

\bibitem{tennis} A.~de Mier and M.~Noy, A solution to the tennis ball
  problem, \emph{Theoret.\ Comput.\ Sci.} \textbf{346} (2005)
  254--264.

\bibitem{lpmcomput} J.~Morton and J.~Turner, Computing the Tutte
  polynomial of lattice path matroids using determinantal circuits,
 \emph{Theoret.\ Comput.\ Sci.}\ \textbf{598} (2015) 150--156.

\bibitem{oxley} James G.~Oxley, \emph{Matroid Theory}, second edition,
  Oxford University Press, Oxford, (2011).

\bibitem{MR2802156} J.~Schweig, Toric ideals of lattice path matroids
  and polymatroids, \emph{J.\ Pure Appl.\ Algebra} \textbf{215} (2011)
  2660--2665.

\bibitem{MR2578897} J.~Schweig, On the {$h$}-vector of a lattice path
  matroid, \emph{Electron.\ J. Combin.}\ \textbf{17} (2010) Note 3, 6pp.


\bibitem{MR843383} E.~Tardos, Generalized matroids and supermodular
  colourings, in \emph{Matroid theory (Szeged, 1982)} North-Holland,
  Amsterdam, (1985) 359--382.

\end{thebibliography}
\end{document}